\documentclass[11pt]{amsart} 
\usepackage[lmargin=1in,rmargin=1in,tmargin=1in,bmargin=1in]{geometry}
\usepackage[ps,all,arc,rotate]{xy}
\usepackage{graphicx, float, epstopdf}
\usepackage{bbm}
\usepackage{hyperref, color}
\usepackage{centernot}
\usepackage{fancyhdr}
\usepackage{multirow}
\usepackage[utf8]{inputenc}
\usepackage{amsfonts,amssymb,amsmath,amsthm,mathrsfs}
\usepackage{graphics, setspace}
\usepackage{braket}
\usepackage{mathtools}
\usepackage{tikz}  
\usepackage{upgreek}

\numberwithin{equation}{section}
\numberwithin{figure}{section}
\allowdisplaybreaks[4]         
 
\newtheorem{lemma}{Lemma}[section]
\newtheorem{theorem}{Theorem}[section]
\newtheorem{proposition}{Proposition}[section]

\newtheorem{conjecture}[lemma]{Conjecture}
\newtheorem{corollary}[lemma]{Corollary}
\theoremstyle{definition}
\newtheorem{definition}{Definition}[section]
\newtheorem{remark}{Remark}[section]
\usepackage{booktabs}
\usepackage{array}
\usepackage{vcell}
\newcommand{\dist}{\text{dist}}
\newcommand{\thrzf}{\text{th}_{\text{rzf}}}
\newcommand{\ecc}{\text{ecc}}
\newcommand{\Var}{\text{Var}}
\newcommand{\Geom}{\text{Geom}}
\newcommand{\rad}{\operatorname{rad}}

\newcommand{\Prob}{\mathbb{P}}

\newcommand{\E}{\operatorname{E}}

\newcommand{\mycomment}[1]{}

\newtheorem{thm}{Theorem}[section]
\newcommand{\eptrzf}{\operatorname{ept_{rzf}}}
\newcommand{\eptpzf}{\operatorname{ept_{pzf}}}

\begin{document}

\title[Randomized zero forcing]{Randomized zero forcing}
\author[J.~Geneson]{Jesse Geneson}
\author[I.~Hicks]{Illya Hicks}
\author[N.~Lichtenberg]{Noah Lichtenberg}
\author[A.~Moon]{Alvin Moon}
\author[N.~Robles]{Nicolas Robles}
\date{}

\begin{abstract}
    We introduce randomized zero forcing (RZF), a stochastic color-change process on directed graphs in which a white vertex turns blue with probability equal to the fraction of its incoming neighbors that are blue. Unlike probabilistic zero forcing, RZF is governed by in-neighborhood structure and can fail to propagate globally due to directionality. The model extends naturally to weighted directed graphs by replacing neighbor counts with incoming weight proportions. We study the expected propagation time of RZF, establishing monotonicity properties with respect to enlarging the initial blue set and increasing weights on edges out of initially blue vertices, as well as invariances that relate weighted and unweighted dynamics. Exact values and sharp asymptotics are obtained for several families of directed graphs, including arborescences, stars, paths, cycles, and spiders, and we derive tight extremal bounds for unweighted directed graphs in terms of basic parameters such as order, degree, and radius. We conclude with an application to an empirical input-output network, illustrating how expected propagation time under RZF yields a dynamic, process-based notion of centrality in directed weighted systems.
\end{abstract}

\maketitle

\section{Introduction}

Color changing is a graph-theoretic concept which models the spread of a property,
represented by the color of a vertex, through a graph over discrete time steps.
The concept was formalized in \cite{AIMWG} as a deterministic process known as
\textit{zero forcing}. In zero forcing, a blue vertex $u$ of an arbitrary graph
$G$ changes the color of an adjacent white vertex $w$ if and only if $w$ is the
only white vertex adjacent to $u$. Determining whether an initial set of blue
vertices of $G$ will eventually turn the graph entirely blue is a central
question, with extensive work devoted to characterizing such sets of minimum
size and their relationship to structural and matrix-theoretic parameters of
graphs \cite{hogben_survey}.

\medskip 

The study of zero forcing on graphs was initially motivated by problems in bounding the minimum rank of graphs \cite{AIMWG}, as well as independently arising in quantum control theory \cite{Burgarth_2007Quantum, Severini_2008}. Since its introduction, zero forcing has been extended in many directions by
modifying either the forcing rule or the class of allowable forcing vertices.
Positive semidefinite zero forcing was introduced to study the minimum positive
semidefinite rank problem and differs from standard zero forcing in how forces
propagate across components \cite{psd_zf}. Skew zero forcing further relaxes the
forcing rule by allowing white vertices to force, leading to different extremal
behavior and complexity questions \cite{skewth}. A complementary line of work
focuses on propagation time rather than feasibility, including variants that study tradeoffs between the size of an initial forcing set and the
time required for complete propagation \cite{butler}. Related ideas also
appear in power domination, which models monitoring processes in networks and
can be viewed as a forcing process with delayed observation \cite{power_dom}.  In this setting, electrical networks are monitored using phasor measurement units (PMUs), and the placement of these devices is closely related to the zero forcing process on the underlying network graph. As a result, power domination has been extensively studied across a variety of graph families and graph products \cite{powerdominationproductgraphs}.

\medskip

In \cite{ky_pzf}, Kang and Yi modified the zero forcing definition to be probabilistic. In this new rule, called \textit{probabilistic zero forcing} (PZF), each blue vertex $u$ of an undirected graph $G$ attempts to independently change the color of any incident white vertex $w$ and succeeds with probability, 

\begin{align}\label{def:pzf}
P(\textnormal{$u$ turns $w$ blue}) = \frac{\textnormal{number of blue vertices which are adjacent to $u$}}{\textnormal{degree of $u$}}.
\end{align}

Unlike deterministic zero forcing, minimal forcing sets for the probabilistic
color-change rule \eqref{def:pzf} are trivial: on any connected graph, starting
from any nonempty initial blue set, the process colors the entire graph blue
with probability $1$ \cite{ky_pzf}.  Consequently, the natural parameters in
probabilistic zero forcing (PZF) concern \emph{how quickly} the all-blue state is
reached rather than whether it can be reached.  This motivates the study of the
\emph{expected propagation time} $\eptpzf(G,X)$, defined as the expected number
of update rounds required for PZF to color all of $V(G)$ blue starting from an
initial blue set $X$.

\medskip

The systematic study of $\eptpzf$ was initiated by Geneson and Hogben
\cite{ghpzf}, who established general bounds, exact values on several graph
families, and extremal constructions.  Algorithmic and computational approaches
based on Markov chains were developed in \cite{pzf_using_markov}.  Subsequent
work has sharpened general upper bounds for $\eptpzf$ in terms of classical graph
parameters (such as order and radius) \cite{narayanansun_ejc}, and has developed
asymptotic results for additional families, including random graphs
\cite{english_random} and structured graphs such as grids, regular graphs, and
hypercubes \cite{husun_bica}.  Further refinements and related propagation-time
parameters for these families have also been studied (e.g., in tight-asymptotic
regimes for hypercubes and grids) \cite{behague_eljc}.  Collectively, this body
of work emphasizes that while PZF guarantees eventual propagation on connected
graphs, the expected time scale captures sensitive structural information about
the underlying topology.

\medskip 

In this article, we consider a variation of PZF on directed graphs (edges can be bidirectional) based on a color change rule determined by the in-degree of uncolored vertices. Precisely, suppose $w$ is a white vertex in a directed graph $G$. Then, at each step of the color change process, the probability that $w$ changes from white to blue is given by, 
\begin{align}\label{eq:rzf}
P(\textnormal{$w$ turns blue}) = \frac{\textnormal{\# of incident edges to $w$ with blue source}}{\textnormal{ in-degree of $w$}}.
\end{align}
In the case that $w$ has indegree $0$, the probability that $w$ turns blue is $0$. The color change rule in equation (\ref{eq:rzf}) does not imply that eventually every vertex of $G$ will turn blue, even if $G$ as an undirected graph is connected. If it is possible, then, just as in the case of PZF, the time it takes to color all of $G$ is a random variable determined by the initial set of blue vertices. We define the expected propagation time for this new process, which we call \textit{randomized zero forcing}, to accommodate the chance that the graph is never completely colored. 

We also define randomized zero forcing on graphs with weighted edges by scaling the
spread probability by the relative weight of each incoming edge.
 In an input-output network, this means that a node is more likely to be affected if a sizable share of its required inputs is sourced from a neighbor that has already become infected, capturing the idea that greater economic dependence corresponds to greater vulnerability.
\[
P(w \text{ turns blue}) = 
\frac{\sum_{u \in B^-(w)} w_{uw}}{\sum_{u \in N^-(w)} w_{uw}}
\]
where \(N^-(w)\) denotes the in-neighbors of \(w\) and \(B^-(w)\) denotes the blue in-neighbors of $v$. In the case that $\sum_{u \in N^-(w)} w_{uw} = 0$, the probability that $w$ turns blue is $0$.

\begin{definition}[Randomized zero forcing process]
\label{def:rzf-process}
Let $G$ be a directed graph with nonnegative edge weights $\{w_{uv}\}$, and let
$B_0 \subseteq V(G)$ be the initial set of blue vertices.  
For each integer $t \ge 0$, suppose the current blue set is $B_t$.

For each white vertex $w \notin B_t$, define
\[
p_t(w)
\;:=\;
\begin{cases}
\dfrac{\sum_{u \in B_t} w_{uw}}{\sum_{u \in V(G)} w_{uw}}, & \text{if } \sum_{u \in V(G)} w_{uw} > 0,\\[1.2ex]
0, & \text{if } \sum_{u \in V(G)} w_{uw} = 0.
\end{cases}
\]

At round $t+1$, each white vertex $w \notin B_t$ independently becomes blue with
probability $p_t(w)$, and blue vertices remain blue forever.  That is,
\[
B_{t+1}
=
B_t \;\cup\;
\bigl\{\, w \notin B_t : \text{$w$ succeeds in an independent Bernoulli trial with
parameter $p_t(w)$} \,\bigr\}.
\]

The resulting sequence $(B_t)_{t\ge 0}$ is a time-inhomogeneous Markov chain on
subsets of $V(G)$ with absorbing states.
\end{definition}

\begin{definition}
Let $X \subset V_G$ be a subset of vertices of a directed graph $G$. If it is possible to color all of $G$ blue starting with the initial set of $X$ as the only blue vertices, then $\eptrzf(G,X)$, called the \textit{expected propagation time (for RZF)} is the expected value of the number of iterations it takes to color all of $G$ blue. If it is not possible to color all of $G$ blue starting with $X$, then $\eptrzf(G,X)=\infty$. We use the shorthand $\eptrzf(G,v) = \eptrzf(G,\left\{v\right\})$ for $v \in V(G)$. We let $\eptrzf(G) = \min_{v \in V(G)} \eptrzf(G, v)$. If $\eptrzf(G, v) = \infty$ for all $v \in V(G)$, then $\eptrzf(G) = \infty$. 
\end{definition}

Our motivation for this variation of PZF is to model the propagation of risk in
directed production networks.  We view each vertex as a production unit whose
output depends on inputs received from upstream suppliers, represented by
directed edges.  A vertex is colored blue once it experiences a disruptive
event-such as a supply failure, quality defect, or regulatory shock-that
renders its output compromised.  Crucially, vulnerability is treated as a
\emph{recipient-side} phenomenon: a production unit becomes exposed not because
a particular supplier actively transmits risk, but because a sufficiently large
fraction of its required inputs originate from already-compromised suppliers.

A simple toy model makes this precise.  Suppose that at each discrete time step,
a firm independently samples one of its required inputs, with the probability
of selecting supplier $u$ proportional to the weight of the edge $(u,w)$ (for
example, the share of total inputs sourced from $u$).  If the sampled input
comes from a compromised supplier, the firm’s production fails during that
round, and the firm becomes compromised itself.  Under this model, the
probability that a white vertex $w$ turns blue in a given round is exactly the
fraction of its incoming weight that originates from blue vertices, which is
the randomized zero forcing rule in \eqref{eq:rzf}.  The expected propagation
time $\eptrzf$ then quantifies the expected time until a disruption originating
at a small set of firms spreads throughout the network, providing a natural,
process-based measure of systemic fragility.  Although highly stylized, this
mechanism captures a basic asymmetry present in many real production systems:
risk accumulates through dependence on incoming inputs, rather than through
outgoing influence alone.

\medskip
\noindent\textbf{Our Contributions.}
In Section~\ref{sec:basic} we characterize when $\eptrzf(G,S)$ is finite and we prove foundational monotonicity properties for weighted RZF,
showing that enlarging the initial blue set cannot increase $\eptrzf$, and that increasing
weights on edges emanating from initially blue vertices cannot increase $\eptrzf$. In Section~\ref{sec:families} we compute exact values and sharp asymptotics for several
structured families. In particular, we show RZF propagation is deterministic on arborescences and obtain consequences for
complete $k$-ary trees. We also derive
closed-form expressions for weighted stars
and establish exact propagation time on bidirected cycles with uniform incoming totals.

In Section~\ref{sec:extremal_unweighted} we prove extremal bounds for unweighted directed graphs.
We upper-bound $\eptrzf(G,v)$ in terms of the number of directed edges and give constructions showing quadratic-order
growth is achievable. In Section~\ref{sec:operations} we analyze how $\eptrzf$ behaves under graph operations for unweighted directed graphs, including
joining multiple graphs through a single source vertex, yielding upper bounds in terms of the
component propagation times.

In Section~\ref{sec:weighted_extremal} we return to weighted directed graphs, proving sharp upper bounds
in terms of the minimum edge weight and showing sharpness via explicit weighted constructions. Finally, in Section~\ref{sec:data} we apply the method to a real-life input-output sector network,
computing singleton-start EPT values. We conclude in Section~\ref{sec:conclusion} with open problems and
future directions.

\section{Basic results}\label{sec:basic}

Unlike probabilistic zero forcing on connected undirected graphs, randomized zero
forcing on directed graphs need not propagate to all vertices, even when the
underlying undirected graph is connected.  The obstruction is purely
directional: a vertex can only turn blue if it receives positive-weight input
from vertices that are already blue.  This suggests that the question of whether
$\eptrzf(G,S)$ is finite should depend only on the existence of directed paths of
positive weight from the initial blue set $S$, rather than on more delicate
stochastic considerations.  In fact, as the next theorem shows, there is a
complete and purely structural characterization: the expected propagation time
is finite if and only if every vertex is reachable from $S$ in the directed graph
obtained by retaining only the positive-weight edges of $G$.  In particular, no
additional bottlenecks beyond reachability can prevent eventual propagation
under RZF.

\begin{theorem}\label{thm:finite_ept_characterization}
Let $G$ be a weighted directed graph with nonnegative edge weights, and let
$S\subseteq V(G)$ be the initial blue set.  Let $G^+$ denote the directed graph
obtained from $G$ by retaining exactly those directed edges of positive weight.
Then
\[
\eptrzf(G,S)<\infty
\quad\Longleftrightarrow\quad
\text{every vertex of $G$ is reachable from $S$ in $G^+$.}
\]
\end{theorem}

\begin{proof}
Write $R$ for the set of vertices reachable from $S$ in $G^+$.  We first prove
the following invariant:
\begin{equation}\label{eq:reachable_invariant}
B_t\subseteq R \qquad\text{for all }t\ge 0,
\end{equation}
where $B_t$ is the set of blue vertices after $t$ rounds.

To see \eqref{eq:reachable_invariant}, note that $B_0=S\subseteq R$.  Now fix
$t\ge 0$ and assume $B_t\subseteq R$.  If $v\notin R$, then by definition of
reachability there is \emph{no} vertex $u\in R$ with a positive-weight edge
$(u,v)$; otherwise we would have a directed path $S\leadsto u\to v$ in $G^+$,
forcing $v\in R$.  Hence every in-edge of $v$ with positive weight has its tail
in $V(G)\setminus R$.  Since $B_t\subseteq R$, it follows that the total
incoming weight to $v$ from blue vertices at time $t$ is $0$, and therefore the
RZF update probability for $v$ in round $t+1$ is $0$.  Thus no vertex outside
$R$ can become blue, and $B_{t+1}\subseteq R$.  This proves
\eqref{eq:reachable_invariant} by induction.

If $R\neq V(G)$, then \eqref{eq:reachable_invariant} implies $B_t\neq V(G)$ for
every $t$, so the all-blue state is unreachable and hence
$\eptrzf(G,S)=\infty$.

Assume now that $R=V(G)$.  Fix a directed spanning forest $F$ of $G^+$ rooted at
$S$: for each $v\in V(G)\setminus S$, choose a parent $\pi(v)$ such that
$(\pi(v),v)$ is an edge of $G^+$ (hence has positive weight) and such that
$\pi(v)$ lies strictly closer to $S$ along some directed path in $G^+$.
For each $v\notin S$, write
\[
\alpha(v) := w_{\pi(v),v} > 0,
\qquad
W^-(v) := \sum_{u\in V(G)} w_{uv}
\]
for the weight of the chosen parent edge and the total incoming weight at $v$,
respectively.

For each vertex $v$, let $T(v)$ be the (random) first round $t$ such that
$v\in B_t$ (so $T(v)=0$ for $v\in S$).  Note that the propagation time to the
all-blue state is
\[
\tau := \min\{t : B_t=V(G)\} = \max_{v\in V(G)} T(v).
\]

Fix $v\notin S$.  On every round $t\ge T(\pi(v))$, vertex $\pi(v)$ is blue, so
under the weighted RZF rule the conditional probability (given the history up
to time $t$) that $v$ turns blue at round $t+1$ satisfies
\[
\Pr\!\bigl(v\in B_{t+1}\mid \mathcal{F}_t\bigr)
\;\ge\;
\frac{\alpha(v)}{W^-(v)} \;=: p(v),
\]
with the convention $p(v)=1$ if $W^-(v)=\alpha(v)$.
Since the random updates at $v$ in successive rounds use independent coin flips
(Definition~\ref{def:rzf-process}), it follows that conditional on
$T(\pi(v))$, the additional waiting time $T(v)-T(\pi(v))$ is stochastically
dominated by a geometric random variable with success probability $p(v)$.
In particular,
\begin{equation}\label{eq:Tv-increment-bound}
\E\!\bigl[T(v)-T(\pi(v))\bigr] \;\le\; \frac{1}{p(v)} \;=\; \frac{W^-(v)}{\alpha(v)} < \infty.
\end{equation}

Now fix $v\notin S$ and follow the unique forest path
$v_0 \to v_1 \to \cdots \to v_k=v$ in $F$, where $v_0\in S$ and
$v_{i-1}=\pi(v_i)$ for each $i\ge 1$.  Then
\[
T(v)
=
\sum_{i=1}^k \bigl(T(v_i)-T(v_{i-1})\bigr).
\]
Taking expectations and applying \eqref{eq:Tv-increment-bound} term-by-term gives
\begin{equation}\label{eq:Tv-bound}
\E[T(v)]
\;\le\;
\sum_{i=1}^k \frac{W^-(v_i)}{\alpha(v_i)}
\;<\;\infty.
\end{equation}
Thus every vertex has finite expected coloring time.

Finally, since $\tau=\max_v T(v)\le \sum_v T(v)$ pointwise, we obtain
\[
\E[\tau] \;\le\; \sum_{v\in V(G)} \E[T(v)].
\]
The sum on the right is finite because $V(G)$ is finite and each $\E[T(v)]$ is
finite by \eqref{eq:Tv-bound}.  Hence $\E[\tau]<\infty$, i.e.,
$\eptrzf(G,S)<\infty$ when $R=V(G)$.
\end{proof}

Next, we show two basic results about randomized zero forcing on weighted directed graphs. First, we show that adding initial blue vertices cannot increase the expected propagation time. Then, we show that increasing edge weights for edges which start from blue vertices cannot increase the expected propagation time. Using these results, we also obtain analogous results for unweighted directed graphs.

\begin{lemma}
Let $G$ be a directed graph with a nonnegative weight $\omega(u,v)$ assigned to each directed edge $(u,v)$. 
Let $S_1 \subseteq S_2 \subseteq V(G)$ be two initial sets of blue vertices, 
let $T \subseteq V(G)$ be an arbitrary target set of vertices, and let $\ell \ge 0$. 
Define $A(S, T, \ell)$ as the event that all vertices in $T$ are blue after $\ell$ rounds of the weighted randomized zero forcing (RZF) process starting from initial blue set $S$. 
Then: 
\[ 
\Pr\!\big(A(S_1, T, \ell)\big) \;\le\; \Pr\!\big(A(S_2, T, \ell)\big) \,. 
\] 
\end{lemma}

\begin{proof}

We construct a coupling of two weighted RZF processes, one started from $S_1$ (the \emph{first process}) and one from $S_2$ (the \emph{second process}), 
such that at every time $t$ the set of blue vertices in the first process is contained in the set of blue vertices in the second process. 
Denote by $B_t^{(1)}$ and $B_t^{(2)}$ the blue sets at time $t$ in the first and second process, respectively. 
Our goal is to maintain 
\[ 
B_t^{(1)} \subseteq B_t^{(2)} \quad \text{for all } t = 0,1,\dots,\ell. 
\] 
This monotone coupling immediately implies the lemma, because under the coupling, 
whenever all vertices in $T$ are blue in the first process after $\ell$ rounds, they are also blue in the second process. 

At time $0$, we have $B_0^{(1)} = S_1$ and $B_0^{(2)} = S_2$. Since $S_1 \subseteq S_2$ by assumption, the inclusion $B_0^{(1)} \subseteq B_0^{(2)}$ holds. 
Now assume inductively that after $t$ rounds we have $B_t^{(1)} \subseteq B_t^{(2)}$. 
We describe the coupling of the transitions from time $t$ to time $t+1$. 

At time $t$, consider an arbitrary vertex $w \in V(G)$ that is white in at least one of the processes. 
For the next round, the weighted RZF rule specifies that $w$ becomes blue with a probability that depends on its blue in-neighbors at time $t$. 
Let 
\[ 
p_t^{(1)}(w) := \frac{\sum_{x \in N^-(w) \cap B_t^{(1)}} \omega(x,w)}{\sum_{x \in N^-(w)} \omega(x,w)} 
\quad\text{and}\quad 
p_t^{(2)}(w) := \frac{\sum_{x \in N^-(w) \cap B_t^{(2)}} \omega(x,w)}{\sum_{x \in N^-(w)} \omega(x,w)} \,. 
\] 
Because $B_t^{(1)} \subseteq B_t^{(2)}$, we have 
\[ 
p_t^{(1)}(w) \;\le\; p_t^{(2)}(w) 
\] 
for every vertex $w$. 

To couple the updates, we proceed as follows. 
For each vertex $w$ and each time $t$, we sample a single random variable $U_t(w)$, uniformly distributed in $[0,1]$, and use it in both processes. 
We then declare 
\[ 
w \in B_{t+1}^{(i)} \quad\Longleftrightarrow\quad \bigl(w \in B_t^{(i)}\bigr) \;\text{or}\; \bigl(U_t(w) \le p_t^{(i)}(w)\bigr) 
\] 
for $i=1,2$. 
By construction, for each process $i$ and each $w$, the probability that $w$ becomes blue at time $t+1$ (given the past history) is exactly $p_t^{(i)}(w)$, and the updates for distinct vertices use independent uniforms $\{U_t(w)\}_w$. 
Thus each process individually has the correct weighted RZF distribution. 

Moreover, since $p_t^{(1)}(w) \le p_t^{(2)}(w)$ and the same $U_t(w)$ is used in both processes, we have the implication 
\[ 
U_t(w) \le p_t^{(1)}(w) \;\Longrightarrow\; U_t(w) \le p_t^{(2)}(w) \,. 
\] 
Hence, if $w$ is blue at time $t+1$ in the first process and was white at time $t$ in both processes, then $U_t(w) \le p_t^{(1)}(w)$ and therefore $U_t(w) \le p_t^{(2)}(w)$, so $w$ is also blue at time $t+1$ in the second process. 
Together with the inductive hypothesis $B_t^{(1)} \subseteq B_t^{(2)}$, this shows that 
\[ 
B_{t+1}^{(1)} \subseteq B_{t+1}^{(2)} \,. 
\] 

By induction on $t$, it follows that $B_t^{(1)} \subseteq B_t^{(2)}$ for all $t = 0,1,\dots,\ell$. 
In particular, in the coupled probability space, whenever every vertex in $T$ is blue in the first process after $\ell$ rounds, the same is true in the second process. 
Therefore, 
\[ 
\Pr\!\big(A(S_1, T, \ell)\big) \;\le\; \Pr\!\big(A(S_2, T, \ell)\big) \,, 
\] 
as claimed. 
\end{proof}

\begin{corollary}
Let $G$ be a weighted directed graph and let $S_1 \subseteq S_2 \subseteq V(G)$ be two initial blue sets. Then the expected propagation time under randomized zero forcing (RZF) satisfies:
\[
\eptrzf(G, S_1) \;\ge\; \eptrzf(G, S_2).
\]
\end{corollary}
\begin{proof}
We express the expected propagation time as 
\[
  \eptrzf(G, S) \;=\; \sum_{\ell = 0}^\infty \Big(1 - \Pr\!\big(A_G(S, V, \ell)\big)\Big)\,,
\] 
since $1 - \Pr(A_G(S, V, \ell))$ is the probability that the propagation has not finished by time $\ell$. By the previously established result, if $S_1 \subseteq S_2$ then $\Pr\!\big(A_G(S_1, V, \ell)\big) \le \Pr\!\big(A_G(S_2, V, \ell)\big)$ for each $\ell$. It follows that $1 - \Pr(A_G(S_1, V, \ell)) \ge 1 - \Pr(A_G(S_2, V, \ell))$ for all $\ell$. Summing these inequalities over all $\ell$ gives $\eptrzf(G, S_1) \ge \eptrzf(G, S_2)$, as desired.
\end{proof}

\begin{lemma}
Let $G$ and $G'$ be directed graphs on the same vertex set $V$, with possibly weighted edges, and let $S \subseteq V$ be a set of initially blue vertices. 
Assume $G'$ is obtained from $G$ by increasing the weights of some directed edges $(u,w)$ with $u \in S$ (i.e., all edges whose weights change have their tail in $S$; for each such edge $(u,w)$, $W_{G'}(u,w) \ge W_G(u,w)$, and other edge weights remain unchanged). Here $W_H(x,w)$ denotes the weight of edge $(x,w)$ in graph $H$ (with $W_H(x,w)=0$ if no such edge exists). 
For any set $T \subseteq V$ and any number of rounds $\ell \ge 1$, let $A_G(S,T,\ell)$ be the event that all vertices in $T$ are blue after $\ell$ rounds of RZF on graph $G$ starting from initial blue set $S$. 
Then:
\[
\Pr_{G}\!\big(A_G(S,T,\ell)\big)\;\le\;\Pr_{G'}\!\big(A_{G'}(S,T,\ell)\big)\,,
\] 
i.e., increasing edge weights out of initially-blue vertices (without changing $S$) can only increase the probability of eventually forcing all vertices in $T$ blue.
\end{lemma}

\begin{proof}
We couple the RZF processes on $G$ and $G'$ so that the blue set in $G$ is always contained in the blue set in $G'$.

For each $t=0,1,2,\dots,\ell$, let $B_t^G$ and $B_t^{G'}$ denote the set of blue vertices at time $t$ in the RZF processes on $G$ and $G'$, respectively. Note that once a vertex turns blue it remains blue forever in either process, so in particular 
\[ 
S = B_0^G = B_0^{G'} \subseteq B_t^G,\ B_t^{G'} \quad \text{for all } t \ge 0\,.
\] 

At each round $t = 1,2,\dots,\ell$ and for each vertex $w \in V$, we draw an independent uniform random variable $U_t(w) \sim \mathrm{Unif}(0,1)$, and use the same value $U_t(w)$ for the update of $w$ in both graphs. The update rule in graph $G$ is:
\[ 
w \in B_t^G 
\quad\Longleftrightarrow\quad 
\bigl(w \in B_{t-1}^G\bigr)\ \text{or}\ \bigl(U_t(w) \le p_t^G(w)\bigr)\,,
\] 
where 
\[ 
p_t^G(w) 
\;:=\; 
\frac{\sum_{x \in V \cap B_{t-1}^G} W_G(x,w)}{\sum_{x \in V} W_G(x,w)}\,,
\] 
with the convention that $p_t^G(w) = 0$ if the denominator (total in-weight into $w$ in $G$) is zero. The process on $G'$ is defined analogously, using 
\[ 
p_t^{G'}(w) 
\;:=\; 
\frac{\sum_{x \in V \cap B_{t-1}^{G'}} W_{G'}(x,w)}{\sum_{x \in V} W_{G'}(x,w)}\,,
\] 
again with the convention $p_t^{G'}(w) = 0$ if the denominator is zero. Marginally, each process has the correct RZF distribution, since for a fixed history the probability that $w$ becomes blue at time $t$ is exactly $p_t^G(w)$ in $G$ and $p_t^{G'}(w)$ in $G'$.

We now show by induction on $t$ that 
\[ 
B_t^G \;\subseteq\; B_t^{G'} \qquad \text{for all } t = 0,1,\dots,\ell\,.
\] 

\smallskip\noindent 
\emph{Base step ($t=0$).} We have $B_0^G = S = B_0^{G'}$, so $B_0^G \subseteq B_0^{G'}$.

\smallskip\noindent 
\emph{Inductive step.} 
Assume $B_{t-1}^G \subseteq B_{t-1}^{G'}$ for some $t \ge 1$. Fix a vertex $w \in V$, and let us compare its update probabilities in $G$ and $G'$ at time $t$.

Define 
\[ 
b \;:=\; \sum_{\substack{x \in V\\ x \in B_{t-1}^G}} W_G(x,w)\,, 
\qquad 
d \;:=\; \sum_{x \in V} W_G(x,w)\,,
\] 
so that (for $d>0$) 
\[ 
p_t^G(w) \;=\; \frac{b}{d}\,. 
\] 

Now consider the in-neighbors and weights in $G'$. By assumption, $G'$ is obtained from $G$ by increasing weights on edges $(u,w)$ from some $u \in S$. Let 
\[ 
\Delta \;:=\; \sum_{u \in S} \Big( W_{G'}(u,w) - W_G(u,w) \Big)\,,
\] 
the total additional in-weight into $w$ coming from $S$ in $G'$ (note $\Delta \ge 0$). Since edges not from $S$ have the same weight in $G$ and $G'$, we have 
\[ 
\sum_{x \in V} W_{G'}(x,w) \;=\; \sum_{x \in V} W_G(x,w) + \Delta \;=\; d + \Delta\,,
\] 
so that 
\[ 
\deg_{G'}^-(w) \text{ (total in-weight into $w$ in $G'$)} \;=\; d + \Delta\,.
\] 

Each edge $(u,w)$ whose weight increased has $u \in S$, and every such $u$ is blue at every time step in both processes (since $S$ is initially blue and remains blue). In particular, for each such $u$ we have $u \in B_{t-1}^G$ and $u \in B_{t-1}^{G'}$. Thus all the additional in-weight $\Delta$ in $G'$ comes from blue in-neighbors of $w$. Moreover, by the inductive hypothesis $B_{t-1}^G \subseteq B_{t-1}^{G'}$, any in-neighbor of $w$ that is blue in $G$ at time $t-1$ is also blue in $G'$ at time $t-1$. Therefore, the total weight of blue in-neighbors of $w$ in $G'$ at time $t-1$ satisfies 
\[ 
b' \;:=\; \sum_{\substack{x \in V\\ x \in B_{t-1}^{G'}}} W_{G'}(x,w) \;\ge\; b + \Delta\,,
\] 
since it includes all of $b$ (blue neighbors from $G$) plus $\Delta$ (the added weight from blue neighbors in $S$). Hence (when $d + \Delta > 0$), 
\[ 
p_t^{G'}(w) \;=\; \frac{b'}{\,d + \Delta\,} \;\ge\; \frac{b + \Delta}{\,d + \Delta\,}\,. 
\] 

If $d = 0$, then $p_t^G(w) = 0$. In $G'$, either $\Delta = 0$ as well (so still no in-neighbor weight and $p_t^{G'}(w)=0$), or $\Delta > 0$ (meaning $w$ has acquired some in-weight from blue vertices in $S$), in which case $b' = \Delta$ and $p_t^{G'}(w) = \frac{b'}{d+\Delta} = \frac{\Delta}{\Delta} = 1$. In either case we have $p_t^{G'}(w) \ge p_t^G(w)$.

If $d > 0$, then $0 \le b \le d$, and for any $\Delta \ge 0$ we have 
\[ 
\frac{b + \Delta}{\,d + \Delta\,} \;\ge\; \frac{b}{\,d\,}\,. 
\] 
Indeed, 
\[ 
\frac{b + \Delta}{\,d + \Delta\,} \ge \frac{b}{\,d\,} 
\quad\iff\quad 
(b + \Delta)d \ge b(d + \Delta) 
\quad\iff\quad 
bd + \Delta d \ge bd + b\Delta 
\quad\iff\quad 
\Delta(d - b) \ge 0\,,
\] 
which holds because $\Delta \ge 0$ and $d - b \ge 0$. 

Combining both cases, we conclude that for every vertex $w$, 
\[ 
p_t^{G'}(w) \;\ge\; p_t^G(w)\,. 
\] 

Since we use the same $U_t(w)$ in both processes, it follows that 
\[ 
U_t(w) \le p_t^G(w) \quad \Longrightarrow \quad U_t(w) \le p_t^{G'}(w)\,. 
\] 
Therefore, if $w$ is blue at time $t$ in $G$ (and was white at time $t-1$ in both processes), then $U_t(w) \le p_t^G(w)$ implies $U_t(w) \le p_t^{G'}(w)$, so $w$ is also blue at time $t$ in $G'$. Together with the inductive hypothesis $B_{t-1}^G \subseteq B_{t-1}^{G'}$, this implies 
\[ 
B_t^G \;\subseteq\; B_t^{G'}\,.
\] 

By induction, $B_t^G \subseteq B_t^{G'}$ for all $t = 0,1,\dots,\ell$. 

Finally, if all vertices of $T$ are blue in $G$ after $\ell$ rounds (event $A_G(S,T,\ell)$), then $T \subseteq B_\ell^G \subseteq B_\ell^{G'}$, so $A_{G'}(S,T,\ell)$ also occurs. Under our coupling, this means 
\[ 
A_G(S,T,\ell) \;\subseteq\; A_{G'}(S,T,\ell)\,. 
\] 
Therefore, 
\[ 
\Pr_G\!\big(A_G(S,T,\ell)\big) \;\le\; \Pr_{G'}\!\big(A_{G'}(S,T,\ell)\big)\,. 
\] 
This completes the proof.
\end{proof}

\begin{corollary}
    Let $G$ and $G'$ be weighted directed graphs on the same vertex set $V$, and let $S \subseteq V$ be a set of initially blue vertices. 
Assume $G'$ is obtained from $G$ by increasing the weights on some directed edges $(u,w)$ with $u \in S$ (i.e., all edges with increased weights must originate from vertices in $S$). Then, we have $\eptrzf(G, S) \ge \eptrzf(G', S)$.
\end{corollary}

\begin{proof}
    We express the expected propagation time as a sum over tail probabilities. For any graph $H$ on $V$, let $B_\ell^H$ be the set of blue vertices after $\ell$ steps of the RZF process on $H$ (starting from $S$), and let $A_H(S, V, \ell)$ be the event that all vertices in $V$ are blue within $\ell$ steps in $H$. Then:
    \[
    \eptrzf(G, S) = \sum_{\ell=0}^\infty \mathbb{P}_G\!\big(B_\ell^G \neq V\big)
    = \sum_{\ell=0}^\infty \Big(1 - \mathbb{P}_G\!\big(A_G(S, V, \ell)\big)\Big),
    \] 
    and similarly,
    \[
    \eptrzf(G', S) = \sum_{\ell=0}^\infty \mathbb{P}_{G'}\!\big(B_\ell^{G'} \neq V\big)
    = \sum_{\ell=0}^\infty \Big(1 - \mathbb{P}_{G'}\!\big(A_{G'}(S, V, \ell)\big)\Big).
    \]
    By the given inequality, for each $\ell$,
    \[
    \mathbb{P}_G\!\big(A_G(S, V, \ell)\big) \le \mathbb{P}_{G'}\!\big(A_{G'}(S, V, \ell)\big).
    \]
    Subtracting from 1 yields 
    \[
    1 - \mathbb{P}_G\!\big(A_G(S, V, \ell)\big) \ge 1 - \mathbb{P}_{G'}\!\big(A_{G'}(S, V, \ell)\big).
    \]
    Thus the $\ell$th summand of $\eptrzf(G, S)$ is at least the $\ell$th summand of $\eptrzf(G', S)$. Summing over all $\ell$ gives $\eptrzf(G, S) \ge \eptrzf(G', S)$, as desired.
\end{proof}

By restricting to the case when all edges have weight $1$, we obtain the following corollaries for unweighted directed graphs.

\begin{corollary}
Let $G$ be a directed graph. Let $S_1 \subseteq S_2 \subseteq V(G)$ be two initial sets of blue vertices, let $T \subseteq V(G)$ be an arbitrary target set of vertices, and let $\ell \ge 0$. Define $A(S, T, \ell)$ as the event that all vertices in $T$ are blue after $\ell$ rounds of the randomized zero forcing (RZF) process starting from initial blue set $S$. Then:
\[
\Pr\!\big(A(S_1, T, \ell)\big) \;\le\; \Pr\!\big(A(S_2, T, \ell)\big)\,. 
\]
\end{corollary}

\begin{corollary}
Let $G$ be a directed graph and let $S_1 \subseteq S_2 \subseteq V(G)$ be two initial blue sets. Then the expected propagation time under randomized zero forcing (RZF) satisfies:
\[
\eptrzf(G, S_1) \;\ge\; \eptrzf(G, S_2).
\]
\end{corollary}

\begin{corollary}
Let $G$ and $G'$ be directed graphs on the same vertex set $V$, and let $S \subseteq V$ be a set of initially blue vertices. 
Assume $G'$ is obtained from $G$ by adding some directed edges $(u,w)$ with $u \in S$ (i.e., all added edges originate from vertices in $S$). 
For any set $T \subseteq V$ and any number of rounds $\ell \ge 1$, let $A_G(S,T,\ell)$ be the event that all vertices in $T$ are blue after $\ell$ rounds of RZF on graph $G$ starting from initial blue set $S$. 
Then:
\[
\Pr_{G}\!\big(A_G(S,T,\ell)\big)\;\le\;\Pr_{G'}\!\big(A_{G'}(S,T,\ell)\big)\,.
\]
In other words, adding extra edges out of initially-blue vertices (without changing $S$) can only increase the probability of eventually forcing all vertices in $T$ blue.
\end{corollary}

\begin{corollary}
    Let $G$ and $G'$ be directed graphs on the same vertex set $V$, and let $S \subseteq V$ be a set of initially blue vertices. 
Assume $G'$ is obtained from $G$ by adding some directed edges $(u,w)$ with $u \in S$ (i.e., all added edges originate from vertices in $S$). Then, we have $\eptrzf(G, S) \ge \eptrzf(G', S)$.
\end{corollary}

\section{Special families of directed graphs}\label{sec:families}

Before turning to specific graph families, it is useful to identify basic
structural transformations that leave the RZF process unchanged. Such
invariances serve two complementary purposes. First, they substantially reduce
the effective parameter space of weighted directed graphs: many distinct weight
assignments induce exactly the same stochastic dynamics, and hence the same
expected propagation behavior. Second, they provide a principled way to
normalize or simplify weights without loss of generality, allowing results for
unweighted or uniformly weighted graphs to be extended immediately to broader
weighted settings.

In particular, because the RZF update rule depends only on \emph{relative}
incoming weights at each vertex, any transformation that preserves these
ratios leaves the process invariant. Lemma~\ref{lem:incoming_scaling} formalizes
this observation, showing that incoming weights may be rescaled independently
at each vertex without affecting the dynamics. Lemma~\ref{lem:constant_incoming}
then identifies an important special case: when all incoming edges to a vertex
have equal weight, weighted RZF coincides exactly with unweighted RZF on the same
underlying directed graph. Together, these lemmas justify treating many weighted
models as essentially unweighted.

\begin{lemma}\label{lem:incoming_scaling}
Let $G$ be a weighted directed graph. Fix a vertex $w$ with
$\sum_{u\in N^-(w)} w_{uw}>0$.
If all incoming edge weights to $w$ are multiplied by the same constant
$\lambda>0$, then the RZF process on $G$ is unchanged.
\end{lemma}

\begin{proof}
For any blue set $B$,
\[
\frac{\sum_{u\in B} \lambda w_{uw}}{\sum_{u\in V} \lambda w_{uw}}
=
\frac{\sum_{u\in B} w_{uw}}{\sum_{u\in V} w_{uw}}.
\]
Thus the update probability at $w$ is unchanged in every configuration, and
hence the entire RZF process is unchanged.
\end{proof}

\begin{lemma}\label{lem:constant_incoming}
Let $G$ be a directed graph with edge weights satisfying the following
property: for each vertex $w$, there exists a constant $c_w>0$ such that
\[
w_{uw}=c_w \quad \text{for all } u\in N^-(w).
\]
Then weighted RZF on $G$ coincides with unweighted RZF on the same underlying
directed graph.
\end{lemma}

\begin{proof}
For any blue set $B$,
\[
\Pr(w\text{ turns blue})
=
\frac{\sum_{u\in B^-(w)} c_w}{\sum_{u\in N^-(w)} c_w}
=
\frac{|B^-(w)|}{|N^-(w)|}.
\]
\end{proof}

\subsection{Deterministic families}

In this subsection we describe families of weighted directed graphs for which the RZF propagation process is deterministic, so that the expected propagation time coincides with the actual propagation time.

\begin{theorem}
\label{thm:arborescences}
Let $G$ be a weighted directed graph that has a distinguished \emph{root} vertex $r$ with $\deg^-(r)=0$ (no incoming edges), and every other vertex has $\deg^- = 1$ (exactly one parent in $G$). If all vertices of $G$ are reachable from $r$, then the RZF propagation process starting from $r$ is deterministic, and the expected propagation time equals the eccentricity of $r$ in $G$ (i.e., the distance from $r$ to the furthest vertex in the directed tree).
\end{theorem}

\begin{proof}
We show that the process is deterministic and that after round $t$ the blue set
is exactly the set of vertices at directed distance at most $t$ from $r$.

For $t\ge 0$, let
\[
L_t:=\{v\in V(G): \dist(r,v)\le t\},
\]
where $\dist(r,v)$ is the length of the unique directed path from $r$ to $v$ in
the arborescence (well-defined since every non-root vertex has indegree $1$).

We claim that for all $t\ge 0$ we have $B_t=L_t$, where $B_t$ is the set of blue
vertices after $t$ rounds.

Base case $t=0$: $B_0=\{r\}=L_0$.

Inductive step: assume $B_t=L_t$ for some $t\ge 0$. Let $v$ be any vertex with
$\dist(r,v)=t+1$. Let $u$ be the unique in-neighbor of $v$ (its parent). Then
$\dist(r,u)=t$, so $u\in L_t=B_t$. Since $v$ has indegree $1$, its total
incoming weight is exactly the weight on the edge $(u,v)$, and that entire
incoming weight comes from a blue vertex at time $t$. Hence the RZF update rule
gives
\[
\Pr(v\in B_{t+1}\mid B_t)=1.
\]
Thus every vertex in $L_{t+1}\setminus L_t$ becomes blue at round $t+1$, so
$L_{t+1}\subseteq B_{t+1}$.

Conversely, if $\dist(r,v)\ge t+2$, then its parent $u$ satisfies
$\dist(r,u)\ge t+1$, hence $u\notin L_t=B_t$, so $v$ has no blue in-neighbor at
time $t$ and therefore
\[
\Pr(v\in B_{t+1}\mid B_t)=0.
\]
Thus no vertex outside $L_{t+1}$ can turn blue at round $t+1$, giving
$B_{t+1}\subseteq L_{t+1}$.

Therefore $B_{t+1}=L_{t+1}$, completing the induction.

It follows that the propagation time is deterministic and equals the least
$t$ such that $L_t=V(G)$, i.e.\ $\ecc(r)$. Hence $\eptrzf(G,r)=\ecc(r)$.
\end{proof}

\begin{corollary}
\label{cor:kary-tree}
For the weighted complete unidirectional $k$-ary tree of depth $n$, the expected propagation time (starting with the root blue) is exactly $n$. In particular, although each internal node has $k$ children, the propagation process is deterministic and terminates in $n$ rounds.
\end{corollary}

\begin{proof}
The unidirectional $k$-ary tree of depth $n$ is a directed arborescence in which each non-root vertex has indegree 1 and the longest root-leaf path has length $n$. By Theorem~\ref{thm:arborescences}, the propagation finishes after $n$ rounds. (In this tree, every leaf lies at distance $n$ from the root, so the root’s eccentricity is $n$.) 
\end{proof}

\subsection{Stars}

Let $S_m$ be the bidirected star with center $c$ and leaves
$L=\{\ell_1,\dots,\ell_m\}$.  
Assume edges $(\ell_i,c)$ have weights $a_i>0$ and edges $(c,\ell_i)$ have
weights $b_i>0$.

\begin{remark}
If $v$ denotes the center vertex of $S_m$, then $\eptrzf(\overleftrightarrow{K_{1,n-1}},v) = 1$.
\end{remark}

\begin{theorem}\label{thm:weighted_star}
Let $S\subseteq L$ be a nonempty set of leaves, and assume that $c$ is not
initially blue. Then
\[
\eptrzf(S_m,S)
=
1+\frac{\sum_{i=1}^m a_i}{\sum_{\ell_i\in S} a_i}.
\]
In particular, for a single leaf $\ell_j$,
\[
\eptrzf(S_m,\{\ell_j\})
=
1+\frac{\sum_{i=1}^m a_i}{a_j}.
\]
\end{theorem}

\begin{proof}
While $c$ is white, no leaf outside $S$ has a blue in-neighbor, so the blue
set remains exactly $S$. In each round, the probability that $c$ turns blue is
\[
p=\frac{\sum_{\ell_i\in S} a_i}{\sum_{i=1}^m a_i},
\]
which is constant across rounds. Hence the time $T_c$ until $c$ becomes blue
is geometric with $\E[T_c]=1/p$.

Once $c$ becomes blue, every remaining leaf $\ell_i\notin S$ has exactly one
in-neighbor, namely $c$, and therefore turns blue with probability $1$ in the
next round. Thus
\[
\eptrzf(S_m,S)=\E[T_c]+1
=
1+\frac{\sum_{i=1}^m a_i}{\sum_{\ell_i\in S} a_i}.
\]
\end{proof}

\begin{corollary}\label{cor:best_leaf}
Among all single-leaf starting vertices, $\eptrzf(S_m,\{\ell_j\})$ is minimized
when $a_j$ is maximal.
\end{corollary}

\begin{proof}
Immediate from Theorem~\ref{thm:weighted_star}.
\end{proof}

\begin{corollary}
If $S_m$ has all edges of equal weight and $v$ is a leaf vertex of $S_m$, $\eptrzf(S_m,v) = m+1$.
\end{corollary}

\subsection{Paths}

We next analyze the RZF propagation process on weighted bidirectional paths, where the linear structure permits an exact computation of the expected propagation time.

\begin{thm}
Let $P$ be a bidirectional path of order $n$ with endpoints $v_1, v_n$. Suppose that there are edges from $v_i$ to $v_{i+1}$ and edges from $v_{i+1}$ to $v_{i}$ for all $1 \le i \le n-1$. We denote weight on the edge from $v_{i-1}$ to $v_{i}$ as $l_i$ and the edge from $v_{i+1}$ to $v_i$ as $r_i$, then \[\eptrzf(P, v_1) = \sum_{i = 2}^{n-1} \frac{l_i+r_i}{l_i}+1\] and \[\eptrzf(P, v_n) = \sum_{i = 2}^{n-1} \frac{l_i+r_i}{r_i}+1\].
\end{thm}
\begin{proof}
We prove the statement for the start vertex $v_1$; the case $v_n$ is symmetric.

For $2\le i\le n$, let $T_i$ be the (random) number of additional rounds needed
to turn $v_i$ blue after $v_{i-1}$ has become blue, assuming that none of
$v_i,\dots,v_n$ is blue yet. We first justify that this describes the process:
on a bidirected path starting from $v_1$, at any time before $v_i$ turns blue,
the only vertex among $\{v_i,\dots,v_n\}$ that can have a blue in-neighbor is
$v_i$ itself (its only possible blue in-neighbor is $v_{i-1}$). Hence the
process colors vertices in the order $v_1,v_2,\dots,v_n$ almost surely.

Fix $2\le i\le n-1$. Once $v_{i-1}$ is blue and $v_i$ is still white, the blue
in-weight into $v_i$ equals the weight on $(v_{i-1},v_i)$, namely $l_i$, while
the total in-weight into $v_i$ equals $l_i+r_i$. Therefore in each subsequent
round, independently across rounds, $v_i$ becomes blue with probability
\[
p_i=\frac{l_i}{l_i+r_i}.
\]
Thus $T_i\sim \Geom(p_i)$ and $\E[T_i]=1/p_i=(l_i+r_i)/l_i$.

For $i=n$, the endpoint $v_n$ has indegree $1$ (its only in-neighbor is
$v_{n-1}$), so once $v_{n-1}$ is blue we have $\Pr(v_n\text{ turns blue next
round})=1$, hence $\E[T_n]=1$.

By linearity of expectation,
\[
\eptrzf(P,v_1)=\E[T_2+\cdots+T_n]
=\sum_{i=2}^{n-1}\frac{l_i+r_i}{l_i}+1,
\]
as claimed.
\end{proof}

\begin{corollary} \label{thm:bi_path}
    If $P$ is a bidirectional path graph with all weights equal and $v$ is an endpoint of $P$, then $\eptrzf(P,v)=2n-3$.
\end{corollary}

\subsection{Cycles}

Let $\overleftrightarrow{C_n}$ be the bidirected cycle on vertices
$v_0,v_1,\dots,v_{n-1}$ (indices modulo $n$). Assume each clockwise edge
$(v_i,v_{i+1})$ has weight $p>0$ and each counterclockwise edge $(v_i,v_{i-1})$
has weight $q>0$. Thus every vertex has total in-weight $p+q$.

\begin{theorem}\label{thm:weighted_cycle_rec}
Let the initial blue set consist of $k\ge 1$ consecutive vertices in
$\overleftrightarrow{C_n}$. Then the expected propagation time satisfies
\[
\eptrzf(\overleftrightarrow{C_n},B_0)= n-k.
\]
In particular, the expected propagation time is independent of the weights
$p$ and $q$.
\end{theorem}

\begin{proof}
Because the initial blue set is a contiguous arc on the cycle and a white
vertex can only turn blue if it has at least one blue in-neighbor, the blue
set remains a single contiguous block for all time. Let $X_t$ be the number of
blue vertices after round $t$; then $(X_t)$ is a Markov chain on
$\{k,k+1,\dots,n\}$ with absorption at $n$.

Fix a state $m$ with $k\le m\le n-2$. Exactly two white vertices are adjacent
to the blue block (one on each side). Call them $w_L$ and $w_R$. Vertex $w_L$
has exactly one blue in-neighbor entering it, and that in-edge has weight $p$
(or $q$ depending on orientation); similarly $w_R$ has exactly one blue
in-neighbor entering it with the opposite weight. Concretely, one boundary
white vertex turns blue with probability
\[
a:=\frac{p}{p+q},
\qquad
\text{and the other turns blue with probability}
\qquad
b:=\frac{q}{p+q}.
\]
Under the RZF rule, distinct white vertices update independently each round,
so the events $\{w_L\text{ turns blue}\}$ and $\{w_R\text{ turns blue}\}$ are
independent.

Hence, writing $r:=ab=\dfrac{pq}{(p+q)^2}$, the one-step transitions from state
$m$ are:
\[
\Pr(X_{t+1}=m)=\Pr(\text{neither boundary turns})=(1-a)(1-b)=ab=r,
\]
\[
\Pr(X_{t+1}=m+2)=\Pr(\text{both boundaries turn})=ab=r,
\]
\[
\Pr(X_{t+1}=m+1)=1-2r.
\]
When $m=n-1$, the unique remaining white vertex has both in-neighbors blue, so
it turns blue with probability $1$, i.e. the chain moves from $n-1$ to $n$ in
one round deterministically.

Now let $T_m$ denote the expected additional number of rounds to absorption
(starting from $X_t=m$). Then $T_n=0$, $T_{n-1}=1$, and for $k\le m\le n-2$ the
law of total expectation gives the standard hitting-time recurrence
\[
T_m
=
1
+\Pr(m)T_m
+\Pr(m+1)T_{m+1}
+\Pr(m+2)T_{m+2},
\]
i.e.
\begin{equation}\label{eq:Tm-recur}
T_m
=
1
+rT_m
+(1-2r)T_{m+1}
+rT_{m+2}.
\end{equation}
Equivalently,
\[
(1-r)T_m
=
1
+(1-2r)T_{m+1}
+rT_{m+2}.
\]

We claim that
\[
T_m = n-m\qquad \text{for all } m\in\{k,k+1,\dots,n\}.
\]
This is immediate for $m=n$ and $m=n-1$. For $k\le m\le n-2$, substitute
$T_m=n-m$, $T_{m+1}=n-(m+1)$, $T_{m+2}=n-(m+2)$ into \eqref{eq:Tm-recur}. The
right-hand side becomes
\[
1+r(n-m)+(1-2r)(n-m-1)+r(n-m-2).
\]
Let $d:=n-m$. Then this equals
\[
1+ r d + (1-2r)(d-1) + r(d-2)
=
1 + \bigl(r+(1-2r)+r\bigr)d - \bigl((1-2r)+2r\bigr)
=
1 + d - 1
=
d,
\]
which matches the left-hand side $T_m=d$. Thus $T_m=n-m$ satisfies the
recurrence and boundary conditions, so it is the correct expected hitting
time.

Therefore, starting from $X_0=k$ we have $\E(\tau)=T_k=n-k$.
\end{proof}

\begin{corollary}
\label{prop:directed-cycle}
Consider a directed cycle graph on $n$ vertices with all edges oriented consistently in one direction (forming a directed cycle) and all weights equal. If one vertex is initially blue, then the propagation time is exactly $n-1$ (deterministic). 
\end{corollary}

\begin{corollary}
\label{prop:bidirected-cycle}
Consider a bidirectional cycle graph on $n$ vertices with all weights equal. If one vertex is initially blue, then the propagation time is exactly $n-1$. 
\end{corollary}

\subsection{Balanced spiders}

In this subsection we restrict attention to spider graphs whose legs all have the same length, which allows a clean asymptotic analysis of the RZF propagation time.

\begin{thm}
Fix an integer $k > 1$ and let $S$ be a bidirectional spider graph with $k$ legs each of length $n$ and all weights equal. Then $\eptrzf(S) = 2n(1+o(1))$, where the $o(1)$ is with respect to $n$.
\end{thm}

\begin{proof}
Let $c$ be the center.  For each leg $i\in\{1,\dots,k\}$, label its vertices by
$v_{i,1},v_{i,2},\dots,v_{i,n}$, where $v_{i,1}$ is adjacent to $c$ and $v_{i,n}$
is the leaf.  All edges have weight $1$, so a white vertex turns blue in a round
with probability equal to the fraction of its in-neighbors that are blue.

\medskip
\noindent\emph{Step 1 (a leg is a sequential process).}
Assume $c$ is blue.  Fix a leg $i$ and suppose $v_{i,1},\dots,v_{i,j-1}$ are blue
and $v_{i,j},\dots,v_{i,n}$ are white at the start of some round (where $j\ge 1$).
Then $v_{i,j}$ has a blue in-neighbor (namely $c$ if $j=1$, and $v_{i,j-1}$ if
$j\ge 2$), while every $v_{i,\ell}$ with $\ell>j$ has all in-neighbors among
$\{v_{i,j},v_{i,j+1},\dots,v_{i,n}\}$, hence has no blue in-neighbor.  Therefore,
on each leg the vertices become blue in order $v_{i,1},v_{i,2},\dots,v_{i,n}$
almost surely.  Moreover, for $1\le j\le n-1$ the vertex $v_{i,j}$ has exactly
two in-neighbors and exactly one is blue when it first becomes eligible, so it
turns blue in each subsequent round with probability $1/2$, independently across
rounds.  Finally $v_{i,n}$ has indegree $1$, so once $v_{i,n-1}$ is blue it
becomes blue in the next round with probability $1$.

Thus the time $T_i$ to complete leg $i$ from the start state $\{c\}$ satisfies
\[
T_i \;=\; \sum_{j=1}^{n-1} G_{i,j} \;+\; 1,
\]
where the $G_{i,j}$ are independent $\Geom(1/2)$ random variables (supported on
$\{1,2,\dots\}$).  In particular,
\[
\E[T_i]=2(n-1)+1=2n-1,
\qquad
\Var(T_i)=(n-1)\Var(\Geom(1/2))=2(n-1).
\]

\medskip
\noindent\emph{Step 2 (upper bound from starting at the center).}
Let $\tau$ be the time to turn all vertices blue starting from $\{c\}$.  Since
the spider is all-blue iff every leg is complete,
\[
\tau=\max_{1\le i\le k} T_i,
\qquad\text{so}\qquad
\eptrzf(S)\le \eptrzf(S,\{c\})=\E[\tau].
\]
Fix $t>0$.  By Chebyshev and a union bound,
\[
\Pr(\tau\ge (2n-1)+t)
\le
\sum_{i=1}^k \Pr(T_i-(2n-1)\ge t)
\le
\sum_{i=1}^k \frac{\Var(T_i)}{t^2}
\le
\frac{2kn}{t^2}.
\]
Using $\E[\tau]=(2n-1)+\int_0^\infty \Pr(\tau\ge (2n-1)+u)\,du$, we obtain
\[
\E[\tau]
\le
(2n-1) + \int_0^\infty \min\!\left\{1,\frac{2kn}{u^2}\right\}\,du
=
(2n-1) + O(\sqrt{n}),
\]
where the implicit constant depends only on $k$.

\medskip
\noindent\emph{Step 3 (lower bound).}
Let $s$ be any starting vertex.  There exists a leaf $\ell$ at (undirected) graph
distance at least $n$ from $s$ (if $s=c$ then every leaf is at distance $n$, and if
$s$ lies on some leg then any leaf on a different leg is at distance $>n$).
Along the unique simple path from $s$ to $\ell$, the blue set must advance across
at least $n$ edges.  Each advance requires at least one successful round, and for
the first $n-1$ advances the next vertex has at most one blue in-neighbor among
two in-neighbors when it first becomes eligible, hence succeeds with probability
at most $1/2$ each round; the final advance (into the leaf) has success probability
$1$.  Therefore $\eptrzf(S,s)\ge 2(n-1)+1=2n-1$, and taking the minimum over $s$
gives $\eptrzf(S)\ge 2n-1$.

Combining the bounds yields $\eptrzf(S)=2n+O(\sqrt{n})$, hence
$\eptrzf(S)=2n(1+o(1))$ as $n\to\infty$ for fixed $k$.
\end{proof}

Indeed, the last proof implies the following slightly stronger result. 

\begin{thm}
Fix an integer $k > 1$ and let $S$ be a bidirectional spider graph with $k$ legs each of length $n(1+o(1))$ and all weights equal. Then $\eptrzf(S) = 2n(1+o(1))$, where the $o(1)$ is with respect to $n$.
\end{thm}

In the case that $k = 2$, we obtain the following corollary for $P_n$. Note that $P_n$ can be considered a spider with two legs each of length approximately $n/2$.

\begin{corollary}
For bidirectional paths of order $n$ with all weights equal, we have $\eptrzf(\overleftrightarrow{P_n}) = n(1+o(1))$.
\end{corollary}

\begin{table}[h]
\centering
\begin{tabular}{|c|c||c|c||c|c||c|c|}
\hline
$n$ & $\eptrzf(P_n)$ & $n$ & $\eptrzf(P_n)$ & $n$ & $\eptrzf(P_n)$ & $n$ & $\eptrzf(P_n)$ \\
\hline
4  & 3.0000   & 20 & 20.4216  & 36 & 37.3090  & 52 & 54.0050 \\
5  & 3.6667   & 21 & 21.3519  & 37 & 38.2595  & 53 & 54.9645 \\
6  & 5.2222   & 22 & 22.5494  & 38 & 39.4038  & 54 & 56.0837 \\
7  & 6.0370   & 23 & 23.4836  & 39 & 40.3557  & 55 & 57.0439 \\
8  & 7.4444   & 24 & 24.6711  & 40 & 41.4960  & 56 & 58.1609 \\
9  & 8.3086   & 25 & 25.6086  & 41 & 42.4492  & 57 & 59.1219 \\
10 & 9.6447   & 26 & 26.7875  & 42 & 43.5859  & 58 & 60.2367 \\
11 & 10.5327  & 27 & 27.7279  & 43 & 44.5403  & 59 & 61.1984 \\
12 & 11.8255  & 28 & 28.8993  & 44 & 45.6735  & 60 & 62.3111 \\
13 & 12.7279  & 29 & 29.8421  & 45 & 46.6291  & 61 & 63.2736 \\
14 & 13.9908  & 30 & 31.0069  & 46 & 47.7591  & 62 & 64.3843 \\
15 & 14.9032  & 31 & 31.9520  & 47 & 48.7158  & 63 & 65.3474 \\
16 & 16.1438  & 32 & 33.1109  & 48 & 49.8428  & 64 & 66.4563 \\
17 & 17.0636  & 33 & 34.0579  & 49 & 50.8005  & 65 & 67.4200 \\
18 & 18.2869  & 34 & 35.2115  & 50 & 51.9248  & 66 & 68.5272 \\
19 & 19.2125  & 35 & 36.1603  & 51 & 52.8833  & 67 & 69.4915 \\
\hline
\end{tabular}
\caption{$\eptrzf(P_n)$ for $n = 4$ to $n = 67$, obtained through computation}
\end{table}

\subsection{Other families}

We close this section by examining a collection of additional graph families-complete graphs, stars, Sun graphs, complete bipartite graphs, and $k$-ary trees-that exhibit a range of asymptotic behaviors under the RZF process.

\begin{thm}
    Let $K_{1,n-1}$ denote the undirected unweighted star graph of order $n$, and let $\overleftrightarrow{K_n}$ denote a bidirectional complete graph of order $n$ with all edges of equal weight. If $c$ denotes the center vertex of $K_{1, n-1}$, then for any vertex $v \in \overleftrightarrow{K_n}$ we have $\eptrzf(\overleftrightarrow{K_n},v) = \eptpzf(K_{1,n-1},c) = \Theta(\log{n})$.
\end{thm}

\begin{proof}
Suppose that $v$ in $K_n$ is colored blue and $b$ other vertices are colored blue, so there are $n-b-1$ white vertices. If $w$ is a white vertex, then the probability that $w$ gets colored in the next round of randomized zero forcing is $\frac{b+1}{n-1}$.

Now suppose that the center vertex $c$ in $K_{1,n-1}$ is colored blue and $b$ leaves are also colored blue, so there are $n-b-1$ white leaves. If $w$ is a white leaf, then only $c$ can color $w$ in probabilistic zero forcing, so the probability that $w$ gets colored in the next round is $\frac{b+1}{n-1}$.

Therefore any finite sequence $1=a_1 \le a_2 \le \dots \le a_k = n$ of numbers of blue vertices by round in $K_n$ under randomized zero forcing has the same probability of occurrence as $a_1,\dots,a_k$ has for the number of blue vertices by round in $K_{1,n-1}$ under probabilistic zero forcing. Thus the expected propagation times are the same.  
\end{proof}

\begin{thm}
    For the bidirectional Sun graph $\overleftrightarrow{S_n}$ with all edges of equal weight, we have that
    $$\eptrzf(\overleftrightarrow{S_n}) = 1 + \frac{3}{2}\left(n-1\right).$$
\end{thm}
\begin{proof}
We proceed by induction on $n$.
For the purpose of minimizing expected propagation time, it is optimal to begin at a center vertex, as starting at a leaf delays the first possible forcing of any other vertex by at least one round, whereas starting at the adjacent center forces that leaf after the first iteration.

Let $\eptrzf(S_n^*)$ denote the expected number of rounds required to force all center vertices of $\overleftrightarrow{S_n}$, starting from a single blue center vertex. Note that once all center vertices are blue, exactly one additional iteration suffices to force any remaining leaf vertices.

From a blue center vertex, its attached leaf is forced deterministically in the next round, while each adjacent center vertex is forced independently with probability $\frac{1}{3}$. Thus there is a $\frac{4}{9}$ probability that neither adjacent center vertex is forced, a $\frac{4}{9}$ probability that exactly one is forced, and a $\frac{1}{9}$ probability that both are forced. If exactly $k\in\{0,1,2\}$ adjacent center vertices are forced, the remaining unforced centers form $S_{n-k}^*$. This yields the recurrence
\[
\eptrzf(S_n^*)
= \tfrac{4}{9}\bigl(\eptrzf(S_n^*)+1\bigr)
+ \tfrac{4}{9}\bigl(\eptrzf(S_{n-1}^*)+1\bigr)
+ \tfrac{1}{9}\bigl(\eptrzf(S_{n-2}^*)+1\bigr).
\]

A direct computation gives $\eptrzf(S_3^*)=3$ and $\eptrzf(S_4^*)=\tfrac{9}{2}$. Using the recurrence, one finds $\eptrzf(S_5^*)=6$, and the inductive hypothesis $\eptrzf(S_n^*)=\tfrac{3}{2}(n-1)$ satisfies the recurrence for all $n\ge 5$. Since one additional round suffices to force all leaf vertices,
\[
\eptrzf(\overleftrightarrow{S_n})=\eptrzf(S_n^*)+1
=1+\tfrac{3}{2}(n-1),
\]
as claimed.
\end{proof}

\begin{lemma}\cite{eisenberg} If $X$ is the maximum of $n$ independent copies of $\Geom(p)$, then \[\frac{1}{\ln(1/(1-p))} \sum_{j = 1}^n \frac{1}{j} \le \E(X) \le 1+\frac{1}{\ln(1/(1-p))} \sum_{j = 1}^n \frac{1}{j}.\] 
\end{lemma}

\begin{corollary}\label{maxgeom}
    Fix $p < 1$ and suppose that $X$ is the maximum of $n$ independent copies of $\Geom(p)$. Then, $\E(X) = \Theta(\frac{1-p}{p}\log{n})$. 
\end{corollary}

Define $T_{k, n}$ to be the complete $k$-ary directed tree of depth $n$. This tree has $n+1$ layers, $k^i$ vertices in the $i^{\text{th}}$ layer for each $i = 0, 1, \dots, n$, and edges in both directions between each vertex in layer $i$ and and $k$ vertices in layer $i+1$ for each $i = 0, 1, \dots, n-1$. 

\begin{thm}
If $T_{k, n}$ is a complete $k$-ary directed tree of depth $n$ with all edges of equal weight, then $\eptrzf(T_{k,n}) = O(k n^2)$.
\end{thm} 

\begin{proof}
    Suppose that we start with the depth-$0$ vertex as the initial blue vertex. If the vertices in layer $i$ are all blue, then it takes at most $O(k i)$ additional rounds for the vertices in layer $i+1$ to be colored blue by Corollary~\ref{maxgeom}, where the constant in the $O(i)$ bound depends on $k$. Thus, \[\eptrzf(T_{k,n}) \le \sum_{i = 0}^{n-1} O(k i) = O(k n^2).\]
\end{proof}

\begin{thm}
    If $K_{m,n}$ is the bidirectional complete bipartite graph with parts of size $m$ and $n$ and all edges of equal weight, then $\eptrzf(K_{m,n}) = O(\min(m \log{n}, n \log{m}))$.
\end{thm}

\begin{proof}
    Suppose that the initial blue vertex is in the part of size $m$. By Corollary~\ref{maxgeom}, it takes at most $O(m \log{n})$ expected rounds to color all the vertices in the part of size $n$, and then any remaining uncolored vertices in the graph would be colored deterministically on the next turn. Thus, the expected propagation time would be $O(m \log{n})$. Similarly, if the initial blue vertex is in the part of size $n$, then the expected propagation time would be $O(n \log{m})$. Thus, $\eptrzf(K_{m,n}) = O(\min(m \log{n}, n \log{m}))$.
\end{proof}

\begin{conjecture}
Given a bidirectional complete bipartite graph $K_{a,b}$ with all edges of equal weight where $a > b$, let $A$ denote the partite set with $a$ vertices, and $B$ denote the partite set with $b$ vertices. Given vertices $u \in A$ and $v \in B$,
$$\eptrzf{(K_{a,b}, \{v\})} < \eptrzf{(K_{a,b}, \{u\})}$$
\end{conjecture}

\begin{table}[h]
    \centering
    \begin{tabular}{|cc|cc||cc|cc|}
        \toprule
        $a$ & $b$ & $\eptrzf(K_{a,b}, \{u\})$ & $\eptrzf(K_{a,b}, \{v\})$ &
        $a$ & $b$ & $\eptrzf(K_{a,b}, \{u\})$ & $\eptrzf(K_{a,b}, \{v\})$ \\
        \midrule
        1 & 1 & 1.000000 & 1.000000 & 6 & 4 & 5.671967 & 4.756544 \\
        2 & 1 & 3.000000 & 1.000000 & 6 & 5 & 5.639959 & 5.233268 \\
        2 & 2 & 3.000000 & 3.000000 & 6 & 6 & 5.633242 & 5.633242 \\
        3 & 1 & 4.000000 & 1.000000 & 7 & 1 & 8.000000 & 1.000000 \\
        3 & 2 & 3.936000 & 3.094286 & 7 & 2 & 6.662118 & 3.367973 \\
        3 & 3 & 3.997474 & 3.997474 & 7 & 3 & 6.235815 & 4.191726 \\
        4 & 1 & 5.000000 & 1.000000 & 7 & 4 & 6.076556 & 4.784222 \\
        4 & 2 & 4.712275 & 3.182594 & 7 & 5 & 6.010218 & 5.250222 \\
        4 & 3 & 4.682545 & 4.059473 & 7 & 6 & 5.982660 & 5.637761 \\
        4 & 4 & 4.700136 & 4.700136 & 7 & 7 & 5.973384 & 5.973384 \\
        5 & 1 & 6.000000 & 1.000000 & 8 & 1 & 9.000000 & 1.000000 \\
        5 & 2 & 5.405476 & 3.256923 & 8 & 2 & 7.251845 & 3.408869 \\
        5 & 3 & 5.257313 & 4.111751 & 8 & 3 & 6.676058 & 4.222912 \\
        5 & 4 & 5.221566 & 4.727711 & 8 & 4 & 6.449886 & 4.810055 \\
        5 & 5 & 5.221072 & 5.221072 & 8 & 5 & 6.346981 & 5.268740 \\
        6 & 1 & 7.000000 & 1.000000 & 8 & 6 & 6.296968 & 5.647558 \\
        6 & 2 & 6.049812 & 3.318017 & 8 & 7 & 6.272747 & 5.973450 \\
        6 & 3 & 5.767330 & 4.155198 & 8 & 8 & 6.262249         & 6.262249         \\
        \bottomrule
    \end{tabular}
    \caption{Expected propagation time of bidirectional $K_{a,b}$ with all edges of equal weight, by starting vertex}
    \label{tab:propagation_time_balanced_clean}
\end{table}

\section{Extremal results for unweighted directed graphs}\label{sec:extremal_unweighted}

In this section we study extremal behavior of the expected propagation time for randomized zero forcing on unweighted directed graphs. We obtain sharp upper and lower bounds in terms of basic graph parameters such as the number of edges, order, maximum degree, and radius, and we characterize when these bounds are attained.

\begin{thm}\label{thm:edge_max}
Let $G$ be an unweighted directed graph with $m$ directed edges, and assume every vertex has indegree at least $d\ge 0$.
Fix $v\in V(G)$ with $\eptrzf(G,v)<\infty$.
Then
\[
\eptrzf(G,v)\le m-d.
\]
\end{thm}

\begin{proof}
Let $\tau$ be the (random) propagation time of RZF on $G$ started from $\{v\}$,
so that $\eptrzf(G,v)=\E[\tau]$.

Because $\eptrzf(G,v)<\infty$, Theorem~\ref{thm:finite_ept_characterization}
implies that every vertex of $G$ is reachable from $v$ by a directed path.

For each round $t\ge 0$, let $B_t$ be the (random) set of blue vertices after
$t$ rounds, and let $W_t:=V(G)\setminus B_t$ be the set of white vertices.
Define the potential
\[
\Phi_t \;:=\; \sum_{x\in W_t} \deg^-(x),
\]
the total indegree of the white vertices at time $t$.
Note that $\Phi_t$ is a nonnegative integer-valued random variable,
$\Phi_0=\sum_{x\neq v}\deg^-(x)=m-\deg^-(v)$, and $\Phi_t=0$ if and only if
$W_t=\varnothing$, i.e.\ the process has completed.

For a white vertex $x\in W_t$, let $b_t(x)$ be the number of in-neighbors of $x$
that are blue at time $t$.  Under unweighted RZF, conditional on $B_t$,
vertex $x$ turns blue in the next round with probability $b_t(x)/\deg^-(x)$,
and if it turns blue then $\Phi$ decreases by exactly $\deg^-(x)$.
Therefore, conditional on $B_t$,
\[
\E\!\big[\Phi_t-\Phi_{t+1}\,\big|\,B_t\big]
=
\sum_{x\in W_t} \deg^-(x)\cdot \frac{b_t(x)}{\deg^-(x)}
=
\sum_{x\in W_t} b_t(x).
\]
The final sum counts the number of directed edges from $B_t$ to $W_t$
(each such edge contributes $1$ to exactly one term $b_t(x)$), hence
\[
\E\!\big[\Phi_t-\Phi_{t+1}\,\big|\,B_t\big] \;=\; e(B_t,W_t),
\]
where $e(B_t,W_t)$ denotes the number of directed edges with tail in $B_t$
and head in $W_t$.

We claim that whenever $W_t\neq \varnothing$ (equivalently, $\Phi_t>0$), we have
$e(B_t,W_t)\ge 1$.  Indeed, fix any $w\in W_t$.  Since $w$ is reachable from $v$
and $v\in B_t$, there exists a directed path $v=u_0\to u_1\to\cdots\to u_k=w$.
Let $j$ be the smallest index with $u_j\in W_t$.  Then $j\ge 1$ and
$u_{j-1}\in B_t$, so the edge $u_{j-1}\to u_j$ is a directed edge from $B_t$ to
$W_t$.  Thus $e(B_t,W_t)\ge 1$.

Consequently, for every $t$ with $\Phi_t>0$,
\[
\E\!\big[\Phi_{t+1}\,\big|\,B_t\big]
\;=\;
\Phi_t - \E[\Phi_t-\Phi_{t+1}\mid B_t]
\;\le\; \Phi_t - 1.
\]

Now define the stopping time $\tau:=\min\{t:\Phi_t=0\}$ and the process
\[
M_t \;:=\; \Phi_{\min(t,\tau)} + \min(t,\tau).
\]
The inequality above implies that $(M_t)_{t\ge 0}$ is a supermartingale.
Taking expectations gives $\E[M_t]\le \E[M_0]=\Phi_0$ for all $t$, and so
\[
\E[\min(t,\tau)] \;\le\; \Phi_0.
\]
Letting $t\to\infty$ and using monotone convergence yields
\[
\E[\tau]\;\le\;\Phi_0 \;=\; m-\deg^-(v) \;\le\; m-d,
\]
as claimed.
\end{proof}

\begin{corollary}
Over all unweighted directed graphs $G$ with $m$ directed edges and over all $v \in V(G)$ with $\eptrzf(G,v) < \infty$, the maximum possible value of $\eptrzf(G,v)$ is $m$. Over all directed graphs $G$ with $m$ directed edges in which every vertex has indegree at least $1$ and over all $v \in V(G)$ with $\eptrzf(G,v) < \infty$, the maximum possible value of $\eptrzf(G,v)$ is $m-1$. 
\end{corollary}

\begin{proof}
The upper bounds follow from Theorem~\ref{thm:edge_max}. To see that the first upper bound is sharp, note that $\eptrzf(\overrightarrow{P_{m+1}},v) = m$ when $v$ is the endpoint with indegree $0$. To see that the second upper bound is sharp, we split into two cases. When $m = 2k$ is even, we have $\eptrzf(\overleftrightarrow{P_{k+1}},v) = 2k-1 = m-1$ when $v$ is an endpoint. When $m = 2k+1$ is odd, let $G$ be the graph obtained from $\overleftrightarrow{P_{k+1}}$ with endpoints $u,v$ by adding a new vertex with a single edge from $u$. Then, we have $\eptrzf(G,v) = 2k = m-1$.
\end{proof}

Indeed, the $\overrightarrow{P_{m+1}}$ construction in the last proof implies the following stronger result, since the left endpoint is the only initial blue vertex that can color the whole path. 

\begin{corollary}
Over all unweighted directed graphs $G$ with $m$ directed edges with $\eptrzf(G) < \infty$, the maximum possible value of $\eptrzf(G)$ is $m$. 
\end{corollary}

As a corollary of Theorem~\ref{thm:edge_max}, we obtain an upper bound on the maximum possible expected propagation time for any directed graph of order $n$. 

\begin{corollary}
For all unweighted directed graphs $G$ of order $n$ and for all $v \in V(G)$ with $\eptrzf(G,v) < \infty$, we have $\eptrzf(G,v)= O(n^2)$.
\end{corollary}

\begin{proof}
If $G$ has order $n$, then $G$ has at most $n(n-1)$ directed edges, so $\eptrzf(G,v) \le n^2-n-1$.
\end{proof}

Next, we show that the $O(n^2)$ bound is sharp up to a constant factor.

\begin{thm}\label{thm:omega_n2}
    There exists an unweighted directed graph $G$ of order $n$ such that $G$ has only one vertex $v$ with forward paths to all other vertices and $\eptrzf(G, v) = \Omega(n^2)$.
\end{thm}

\begin{proof}
Consider the directed graph $D = (V,E)$ with $V = \left\{a_1,\dots,a_m,b_1,\dots,b_{m+1}\right\}$ which has edges from $b_i$ to $a_j$ for all $i \ge j$ and from $a_i$ to $b_{i+1}$ for each $i = 1,\dots,m$. The minimum randomized zero forcing set has size $1$, and the only choice is $b_1$, or else it is impossible to color $b_1$. The order of the coloring must be $b_1,a_1,b_2,a_2,\dots$. After $b_i$ becomes blue, the probability that $a_i$ becomes blue is $\frac{1}{m+2-i}$, so the expected number of rounds until $a_i$ becomes blue after $b_i$ becomes blue is $m+2-i$. When $a_i$ becomes blue, $b_{i+1}$ deterministically turns blue in the next round. Thus the expected propagation time is $m+ \sum_{i = 1}^m (m+2-i) = \Omega(m^2)$.
\end{proof}

\begin{corollary}
    Over all unweighted directed graphs $G$ of order $n$ with finite $\eptrzf(G)$, the maximum possible value of $\eptrzf(G)$ is $\Theta(n^2)$.
\end{corollary}

As another corollary of Theorem~\ref{thm:edge_max}, we determine the maximum possible value of $\eptrzf(T,v)$ over trees $T$.

\begin{corollary}
Over all unweighted directed graphs $T$ of order $n$ whose underlying undirected graph is a tree and over all $v \in V(T)$ with $\eptrzf(T,v) < \infty$, the maximum possible value of $\eptrzf(G,v)$ is $2n-3$.
\end{corollary}

\begin{proof}
A directed graph $T$ of order $n$ whose underlying undirected graph is a tree has at most $2(n-1)$ directed edges. Thus $\eptrzf(T,v) \le 2(n-1)-1$ for all $v \in (T)$. To see that this upper bound is sharp, note that $\eptrzf(\overleftrightarrow{P_n},v) = 2n-3$ when $v$ is an endpoint.
\end{proof}

\begin{thm}\label{thm:max-indeg-tail}
Let $G$ be an unweighted directed graph of order $n$ with maximum indegree at most $d$.
If $\eptrzf(G)<\infty$, then
\[
\eptrzf(G)\;\le\; dn-\frac{d(d+1)}{2}.
\]
\end{thm}

\begin{proof}
Fix $v\in V(G)$ with $\eptrzf(G)=\eptrzf(G,v)$ and run RZF from $\{v\}$.
Let $\tau$ be the propagation time and, for each $t\ge 0$,
let $B_t$ be the blue set after $t$ rounds and $W_t:=V(G)\setminus B_t$ the white set.

For $r\in\{1,2,\dots,n-1\}$, define $\sigma_r$ to be the first round $t$ such that
$|W_t|\le r$.  Then $\sigma_{n-1}=0$ and $\sigma_0=\tau$, and
\[
\tau=\sum_{r=1}^{n-1}(\sigma_{r-1}-\sigma_r).
\]
We bound $\E[\sigma_{r-1}-\sigma_r]$ for each $r$.

Fix $r\ge 1$ and condition on the history up to time $\sigma_r$.
If $|W_{\sigma_r}|=0$ there is nothing to prove. Otherwise, since $\eptrzf(G,v)<\infty$,
Theorem~\ref{thm:finite_ept_characterization} implies that there exists at least one
directed edge from $B_{\sigma_r}$ to $W_{\sigma_r}$, and hence there exists a white
vertex $w$ having at least one blue in-neighbor at time $\sigma_r$.

Let $b$ be the number of blue in-neighbors of $w$ at time $\sigma_r$.
Since $w$ has at least one blue in-neighbor, we have $b\ge 1$.
Also, because $|W_{\sigma_r}|\le r$, at most $r-1$ of the in-neighbors of $w$ can be white,
so $b\ge \deg^-(w)-(r-1)$.  Hence
\[
b \;\ge\; \max\{1,\deg^-(w)-(r-1)\}.
\]
  Since $\deg^-(w)\le d$, we have
\[
\Pr(w\text{ turns blue in the next round}\mid \mathcal F_{\sigma_r})
=
\frac{b}{\deg^-(w)}
\;\ge\;
\frac{\max\{1,\deg^-(w)-(r-1)\}}{\deg^-(w)}
\;\ge\;
\frac{1}{r}.
\]
Indeed, if $\deg^-(w)<r$ then $\max\{1,\deg^-(w)-(r-1)\}=1$ and $1/\deg^-(w)\ge 1/r$,
while if $\deg^-(w)\ge r$ then
$\frac{\deg^-(w)-(r-1)}{\deg^-(w)}\ge \frac{1}{r}$ because $(r-1)(\deg^-(w)-r)\ge 0$.

Moreover, as long as $|W_t|\le r$ and $W_t\neq\varnothing$, the same reasoning applies at time $t$:
by Theorem~\ref{thm:finite_ept_characterization} there exists a directed edge from $B_t$ to $W_t$, so there is a white vertex with at least one blue in-neighbor, and the conditional probability that \emph{some} white vertex turns blue in the next round is at least $1/r$.
Therefore, conditional on $\mathcal F_{\sigma_r}$, the waiting time to reduce the
number of white vertices from at most $r$ to at most $r-1$ is stochastically dominated
by a geometric random variable with success probability $1/r$. Hence
\[
\E[\sigma_{r-1}-\sigma_r]\le r.
\]
For $r\ge d+1$ we use the trivial bound
\[
\E[\sigma_{r-1}-\sigma_r]\le d,
\]
because some eligible white vertex has success probability at least $1/d$ in each round.

Combining these bounds and using linearity of expectation,
\begin{align*}
\eptrzf(G)
=\E[\tau]
&\le
\sum_{r=d+1}^{n-1} d \;+\; \sum_{r=1}^{d} r\\
&=
d(n-d-1) + \frac{d(d+1)}{2}
=
dn-\frac{d(d+1)}{2},
\end{align*}
as claimed.
\end{proof}

\begin{corollary}
Over all unweighted directed graphs $G$ of order $n$ with maximum degree $d$ and $\eptrzf(G) < \infty$, the maximum possible value of $\eptrzf(G)$ is $d n -\frac{d(d+1)}{2}$.
\end{corollary}

\begin{proof}
The upper bound $\eptrzf(G) \le d n -\frac{d(d+1)}{2}$ follows from the previous theorem. To see that this upper bound is sharp, consider the graph on $n$ vertices $v_1, v_2, \dots, v_n$ with edges from $v_i$ to $v_{i+1}$ for each $1\le i \le n-1$ and edges from $v_{i+j}$ to $v_i$ for all $j \le d-1$ and $1 \le i \le n-j$. 

The vertex $v_1$ must be the initial blue vertex, since it has indegree $0$. Moreover, the vertices must be colored in the order $v_1, v_2, \dots, v_n$. For each $i = 1, \dots, n-d-1$, it takes $d$ expected rounds to color $v_{i+1}$ after $v_i$ is colored. For each $i = n-d, \dots, n-1$, it takes $n-i$ expected rounds to color $v_{i+1}$ after $v_i$ is colored. Altogether, the expected propagation time is \[d(n-d-1)+ \sum_{i = n-d}^{n-1} (n-i) = d n -\frac{d(d+1)}{2}.\]
\end{proof}

Now, we turn to minimums, starting with all directed graphs $G$ of order $n$.

\begin{thm}
    Over all unweighted directed graphs $G$ of order $n > 1$, the minimum possible value of $\eptrzf(G)$ is $1$.
\end{thm}

\begin{proof}
The lower bound $\eptrzf(G) \ge 1$ follows since $n > 1$. This bound can be attained by $K_{1, n-1}$, so the minimum possible value of $\eptrzf(G)$ is $1$.
\end{proof}

Next, we obtain a sharp lower bound on $\eptrzf(G)$ with respect to $\rad(G)$.

\begin{thm}
    For any unweighted directed graph $G$ and any vertex $v \in V(G)$,  $\eptrzf(G, v) \ge \rad(G)$.
\end{thm}

\begin{proof}Let $v$ be any initial blue vertex in $G$. There must exist some $u$ with forward-distance at least $r = \rad(G)$ from $v$. Thus $u$ at the earliest can only be colored blue in round $r$, so the expected propagation time is at least $r$. 
\end{proof}

The lower bound in the last result is sharp, as evidenced by the unidirectional directed path of order $r+1$.

\begin{thm}
Let $G$ be a directed graph and fix a vertex $v\in V(G)$.  
Let $\rad(G,v)$ denote the directed eccentricity of $v$. Consider randomized zero forcing (RZF) starting from $v$. Then $\eptrzf(G,v) = \rad(G,v)$ if and only if every vertex is reachable from $v$ and every directed cycle of $G$ contains $v$.
\end{thm}
\begin{proof}
We prove each direction separately.

\medskip
\noindent\textbf{($\Rightarrow$)}  
Suppose there exists either a vertex not reachable from $v$, or a directed cycle $C$ that does not contain $v$. In the first case, $\eptrzf(G)>\rad(G,v)$ trivially. Thus assume every vertex is reachable from $v$ and that there exists a directed cycle $C$ disjoint from $v$.

Let $B_t$ denote the blue set after $t$ rounds of RZF.
Because $C$ is disjoint from $v$ and reachable from $v$, there exists a realization of the random forcing choices (of positive probability) in which no vertex of $C$ is colored blue during the first $\rad$ rounds. Thus, the propagation time is greater than $\rad(G,v)$ on this event. Since the event has positive probability, we conclude $\eptrzf(G,v)>\rad(G,v)$.

\medskip
\noindent\textbf{($\Leftarrow$)}  
Assume now that every vertex is reachable from $v$ and that every directed cycle of $G$ contains $v$. Suppose for contradiction that there exists a realization of the RZF process and a vertex $w$ such that
$w\notin B_{\rad(G,v)}$, where $B_{\rad(G,v)}$ denotes the blue vertices of $G$ after the $\rad(G,v)$ round of RZF on this realization. Set $w_0:=w$.

Since $w_0\notin B_{\rad(G,v)}$, there exists an in-neighbor $w_1\in N^-(w_0)$ such that $w_1\notin B_{\rad(G,v)-1}$, as otherwise all in-neighbors of $w_0$ would be blue by round $\rad(G,v)-1$, forcing $w_0$ to be blue at round $\rad(G,v)$.

Iterating this argument, we construct vertices
\[
w_0,w_1,\dots,w_{\rad(G,v)}
\]
such that
\[
w_{i+1}\in N^-(w_i)
\quad\text{and}\quad
w_i\notin B_{\rad(G,v)-i}
\quad\text{for all } i=0,\dots,\rad(G,v)-1.
\]
In particular, $w_{\rad(G,v)}\notin B_0=\{v\}$, so $w_{\rad(G,v)}\neq v$.

This yields a directed walk
\[
w_{\rad(G,v)}\to w_{\rad(G,v)-1}\to\cdots\to w_0=w.
\]

If all vertices in this walk are distinct, then it contains a directed path of length $\rad(G,v)$ ending at $w$ whose initial vertex is not $v$.
Since $v$ reaches $w$ by a directed path of length at most $\rad(G,v)$,
this contradicts the definition of $\rad(G,v)$. Otherwise, some vertex repeats along the walk. Taking the first repetition yields a directed cycle disjoint from $v$, contradicting the assumption that every directed cycle contains $v$. In both cases we obtain a contradiction. Therefore every vertex must be blue by round $\rad(G,v)$, and thus $\eptrzf(G,v) = \rad(G,v)$
\end{proof}

\begin{corollary}
The only unweighted directed graphs $G$ of order $n$ with $\eptrzf(G) = 1$ are those $G$ obtained from $\overrightarrow{K_{1,n-1}}$ by adding any number of edges to the center. There are exactly $n$ such $G$, up to isomorphism.
\end{corollary}

\section{Graph operations for unweighted directed graphs}\label{sec:operations}

In this section, we discuss generally how $\eptrzf$ could change with changes to the underlying graph. Generally, adding vertices or edges will have a broad effect on the expected propagation time. For example, if a white vertex $v$ is added to an unweighted directed graph $G$ such that the in-degree of $v$ is $1$ and the out-degree is $0$, then for any vertex $u$ of $G$ not equal to $v$, 
\begin{align}
    \eptrzf(G,u) \leq \eptrzf(G+v, u)  \leq \eptrzf(G,u)+1.
\end{align}
The upper bound is best possible, since for example, we can let $G$ be the directed graph with underlying undirected graph $K_{1, n}$ having edges from the center $u$ to each leaf, and we can let $v$ have an edge from one of the leaves. In this case, $\eptrzf(G,u) = 1$ and $\eptrzf(G+v, u) = 2$. The lower bound is also best possible, since we can take the same graph $G$ and instead let $v$ have an edge from $u$, in which case $\eptrzf(G+v, u) = \eptrzf(G, u) = 1$. Note that the same bounds also apply in the weighted case.

On the other hand, for the case that a blue vertex $b$ is added to $G$ with in-degree $0$ and out-degree $1$, we have an exact formula for the expected propagation time starting from $b$. 

\begin{proposition}
Let $G$ be any unweighted directed graph. If a blue vertex $b$ is added to $G$ with in-degree $0$ and out-degree $1$, and the incident vertex in $G$ to $b$ is $v$, then
    \[\eptrzf(G+b,b) = d + \eptrzf(G,v),\]
where $d$ is the in-degree of $v$ in $G+b$. 
\end{proposition}

\begin{proof}
    If we color $G+b$ starting from $b$, then the first vertex in $G$ that gets colored must be $v$. It takes $d$ expected rounds until $v$ is colored, and then an additional $\eptrzf(G,v)$ rounds to color the rest of $G$.
\end{proof}

In the following proposition, we extend this last operation to the case that $b$ joins multiple graphs. 

\begin{proposition}\label{prop:4.2}
Let $G_1, \ldots, G_m$ be unweighted directed graphs. Denote by $K$ the joined graph defined in the following way: there is a single vertex $b$ with out-degree $m$ and in-degree $0$ which is incident to a vertex $v_i$ for each $G_i$. Assume further that the $G_i$ have no vertices or edges in common as subgraphs of $K$. 

\medskip 

Let $t_i$ denote the propagation time for the graphs $G_i$ starting with the single blue node $v_i$. Then, 
\begin{align}
\eptrzf(K,b) \leq \max _{i} \eptrzf(G_i, v_i)  + \bigg{(} \frac{ \sum _{i=1}^m \textnormal{Var}(t_i)  }{2} \bigg{)}^{1/2} + O(r m^{1/2}).
\end{align}
where $r= \max \{ r_i\}$ and $r_i$ is the in-degree of $v_i$.
\end{proposition}

\begin{proof}
Let $t_K$ denote the propagation time for $K$ with starting node $b$, so that $\eptrzf(K,b) = \mathbb{E}(t_K)$. By Theorem 2.1 of \cite{Aven1985}, 
\begin{align}
    \mathbb{E}(t_K) & = \mathbb{E}(\max_i \{ t_i + \textnormal{Geom}(1/r_i) \} ) \\ 
    & \leq \max _{i } \mathbb{E}(t_i + \textnormal{Geom}(1/r_i)) + \frac{1}{\sqrt{2}} \bigg{(} \sum_{i=1}^m \textnormal{Var}(t_i + \textnormal{Geom}(1/r_i))   \bigg{)}^{1/2}. \label{eq:three_join}
\end{align}
But $\textnormal{Geom}(1/r_i) $ and $t_i$ are independent random variables for all $i$, and so, 
\begin{align}
    \textnormal{Var}(t_i + \textnormal{Geom}(1/r_i)) & = \textnormal{Var}(t_i) + \frac{1-r_i^{-1}}{r_i^{-2}}.
\end{align}
By subadditivity of the square root function, 
\begin{align}\label{eq:two_join}
    \bigg{(} \frac{ \sum _{i=1}^m \textnormal{Var}(t_i) + \textnormal{Geom}(1/r_i) }{2} \bigg{)}^{1/2} & \leq (m/2)^{1/2} (r^2-r)^{1/2} + \bigg{(} \frac{\sum _{i=1}^m \textnormal{Var}(t_i) }{2} \bigg{)}^{1/2},
\end{align}
where $r = \max \{ r_G, r_H \}$. Lastly, it follows straightforwardly from linearity of expectation that, 
\begin{align}\label{eq:one_join}
    \max_{i} \mathbb{E}(t_i + \textnormal{Geom}(1/r_i)) \leq   \max_{i} \mathbb{E}(t_i ) + r. 
\end{align}
Substituting equations (\ref{eq:two_join}) and (\ref{eq:one_join}) into (\ref{eq:three_join}) gives the inequality. 
\end{proof}

\begin{corollary}\label{thm:joined-paths}
Fix a constant $k$. Consider the unweighted directed graph $G$ obtained by taking $k$ bidirected paths $G_1, \dots, G_k$ (of lengths $n_1,\dots,n_k$ respectively) and joining them at a single initial vertex $x$ that is initially blue. Vertex $x$ has directed out-edges to the starting vertices $v_1,\dots,v_k$ of each path. Then, we have 
\[ 
\eptrzf(G,x) \;=\; 2\max_i\{n_i\} + O(\sqrt{\max_i\{n_i\}})\,,
\] where the constant in the $O$-bound depends on $k$.
\end{corollary}

\begin{proof}
Let $T_i$ denote the propagation time on the path $G_i$ when $v_i$ is the initial blue vertex. By Proposition~\ref{prop:4.2} (applied with $m=k$), we have the bound
\[
\eptrzf(G,x) \;\le\; \max_i\{\,E[T_i]\,\} \;+\; O\!\Big(\sqrt{\max_i\{\Var(T_i)\}}\Big)+O(1),
\] where the constants in the $O$-bounds depend on $k$.

Now, each $G_i$ is a bidirected path with one end fed by a single source. We showed that $E[T_i] = 2n_i - 1$ and $\Var(T_i) = O(n_i)$. Therefore, $\max_i\{E[T_i]\} = 2\max_i\{n_i\} - 1$, and $\max_i\{\Var(T_i)\} = O(\max_i\{n_i\})$.

Substituting these into the bound from Proposition~\ref{prop:4.2} gives 
\[
\eptrzf(G,x) \;\le\; 2\max_i\{n_i\} + O(\sqrt{\max_i\{n_i\}})\,.
\] 
Finally, we have a lower bound of $\eptrzf(G,x) \ge 2\max_i\{n_i\}$, since $E[T_i] = 2n_i - 1$ and it takes at least $1$ round for $v_i$ to get colored. Hence, we conclude that 
\[
\eptrzf(G,x) = 2\max_i\{n_i\} + O(\sqrt{\max_i\{n_i\}})\,,
\] 
as claimed.
\end{proof}

In the next result, we consider the effect of adding a single directed edge between two directed graphs.

\begin{thm}
Let $G, H$ be unweighted directed graphs and suppose that $u \in V(G)$ and $v \in V(H)$, where $v$ has indegree $d$ in $H$. Let $K$ be the directed graph obtained from $G$ and $H$ by adding a single edge from $u$ to $v$. Then, for all $w \in V(G)$, we have \[\eptrzf(K,w) \le 1+d+\eptrzf(G,w)+\eptrzf(H,v).\]
\end{thm}

\begin{proof}
If $w$ is the initial blue vertex, then the expected number of rounds to color all of $G$ is $\eptrzf(G,w)$. Since $v$ has indegree $d+1$ in $K$, it takes $d+1$ expected rounds for $v$ to be colored after $u$ is colored. After $v$ is colored, the expected number of rounds to color all of $H$ is $\eptrzf(H,v)$. Altogether, it takes at most $\eptrzf(G,w)+d+1+\eptrzf(H,v)$ expected rounds to color all of $K$, starting from $w$.
\end{proof}

In the next result, we show that the bound in the last theorem is best possible by exhibiting an infinite family of examples which attain the bound.

\begin{thm}
For all $m$ and $n > d$, there exist unweighted directed graphs $G$ and $H$ and vertices $u, w \in V(G)$ and $v \in V(H)$ where $v$ has indegree $d$ in $H$ such that $\eptrzf(G,w) = m$, $\eptrzf(H,v) = n$, and the directed graph $K$ obtained from $G$ and $H$ by adding an edge from $u$ to $v$ satisfies \[\eptrzf(K,w) = 1+d+\eptrzf(G,w)+\eptrzf(H,v).\]
\end{thm}

\begin{proof}
Let $G = \overrightarrow{P_{m+1}}$ and let $H$ be the graph obtained from $\overrightarrow{P_{n+1}}$ by adding edges from the $d$ rightmost vertices to the left endpoint. Let $w$ be the left endpoint of $G$, $u$ be the right endpoint of $G$, and $v$ be the left endpoint of $H$. Then $\eptrzf(G,w) = m$, $\eptrzf(H,v) = n$, and $\eptrzf(K) = m+(d+1)+n$.
\end{proof}

\section{Further results for weighted directed graphs}\label{sec:weighted_extremal}

On a weighted directed graph $G$, we define a \emph{fort} to be a subset of vertices such that each vertex has at least one in-neighbor in the fort. Given that $G$ is strongly connected, we note that unlike in classical zero forcing, no subgraph can remain unforced forever. We aim to define a RZF analogue where entry into the fort occurs only through rare random events, allowing the expected propagation time to grow arbitrarily large if the starting set of blue vertices do not intersect the fort.

\begin{thm}
    On any graph $G$ with two vertex disjoint forts, there exists a weight assignment of the edges such that $\eptrzf(G)$ is arbitrarily high.
\end{thm}

\begin{proof}
Let $n=|V(G)|$, and fix an arbitrary $k\in\mathbb{N}$. Let $F_1$ and $F_2$ be two vertex-disjoint forts in $G$.

We define an edge-weighting as follows. Assign weight $\varepsilon>0$ to every edge $(u,v)$ with $u\notin F_1$ and $v\in F_1$ (edges \emph{entering} $F_1$), and weight $1$ to every other edge of $G$. We will choose $\varepsilon$ sufficiently small in terms of $k$ and $n$.

Consider a round $t$ such that $F_1\cap B_t=\emptyset$. In order for some vertex of $F_1$ to become blue in round $t+1$, there must exist a blue vertex $u\notin F_1$ that forces some $v\in F_1$ along an entering edge $(u,v)$. For any such $u$, the total outgoing weight from $u$ to white vertices is at least $1$ (since all non-entering edges have weight $1$), while each entering edge from $u$ to $F_1$ has weight $\varepsilon$. Hence, regardless of the configuration outside $F_1$,
\[
\Pr(u \text{ forces a vertex in }F_1 \mid F_1\cap B_t=\emptyset)\le \frac{\deg^+_{F_1}(u)\,\varepsilon}{1}
\le n\varepsilon,
\]
where $\deg^+_{F_1}(u)$ denotes the number of out-neighbors of $u$ in $F_1$, and we used $\deg^+_{F_1}(u)\le n$.

Applying the union bound over all vertices $u$ (at most $n$ choices), we obtain
\[
\Pr(\text{some vertex of }F_1\text{ is forced in round }t+1 \mid F_1\cap B_t=\emptyset)
\le n\cdot (n\varepsilon)=n^2\varepsilon.
\]
Therefore the time $T_{F_1}$ until $F_1$ is first entered is stochastically dominated below by a geometric random variable with success probability at most $n^2\varepsilon$, and thus
\[
\mathbb{E}[T_{F_1}] \ge \frac{1}{n^2\varepsilon}.
\]
Choosing $\varepsilon<\frac{1}{kn^2}$ yields $\mathbb{E}[T_{F_1}]>k$.

Finally, for any initial blue vertex $v$, since $F_1$ and $F_2$ are disjoint, at least one of them does not contain $v$. Apply the above construction to whichever fort does not intersect the initial blue set. In all cases, this produces an edge-weight assignment for which $\eptrzf(G)>k$. Since $k$ was arbitrary, $\eptrzf(G)$ can be made arbitrarily large.
\end{proof}

\subsection{Extremal results}

We now derive sharp extremal bounds for the expected propagation time on weighted directed graphs in terms of the minimum edge weight and the order of the graph.

\begin{thm}
  Suppose that $G$ is any weighted directed graph of order $n$ with all positive edge weights at least $w$, and $v$ is a vertex with forward paths to all other vertices in $G$. Then we have $\eptrzf(G, v) \le 1+\frac{n-2}{w}$.  
\end{thm}

\begin{proof}
Given any configuration with some positive number of blue vertices and some positive number of white vertices, there exists a white vertex $u$ that receives some edge from a blue vertex. Since all edge weights are at least $w$, the probability that $u$ gets colored is at least $w$. Therefore the probability of increasing the number of blues is at least $w$, so the expected time to increase the number of blues is at most $\frac{1}{w}$. Each increase is by at least $1$ and the coloring becomes deterministic if only one uncolored vertex is left, so the expected propagation time is at most $1+\frac{n-2}{w}$.
\end{proof}

By adding weights to the construction in Theorem~\ref{thm:omega_n2}, we can show that the $1+\frac{n-2}{w}$ upper bound is sharp.

\begin{thm}\label{thm:w_att}
    There exists a weighted directed graph $G$ of order $n$ with all positive edge weights at least $w$ such that $G$ has only one vertex $v$ with forward paths to all other vertices and $\eptrzf(G, v) = 1+\frac{n-2}{w}$.
\end{thm}

\begin{proof}
Consider the weighted directed graph on vertices $v_1, v_2, \dots, v_n$ with edges from $v_i$ to $v_{i+1}$ for each $1 \le i < n$ and edges from $v_{i+1}$ to $v_i$ for $2 \le i < n$. The edge from $v_i$ to $v_{i+1}$ for each $1 \le i < n-1$ has weight $w$, while the edge from $v_{n-1}$ to $v_n$ has weight $1$. 

The initial blue vertex must be $v_1$, or else it would never be colored. Once $v_i$ has been colored with $i < n-1$, the expected number of rounds until $v_{i+1}$ is colored is $\frac{1}{w}$. Once $v_{n-1}$ has been colored, $v_n$ is colored in the next round. Thus, $\eptrzf(G, v_1) = 1+\frac{n-2}{w}$.
\end{proof}

\begin{corollary}
    Over all weighted directed graph $G$ of order $n$ with all positive edge weights at least $w$ and over all $v \in V(G)$ with forward paths to all other vertices, the maximum possible value of $\eptrzf(G, v)$ is $1+\frac{n-2}{w}$.
\end{corollary}

\subsection{Operations on weighted directed graphs}

In this subsection, we consider some operations on weighted directed graphs and their effects on expected propagation time. 

\begin{proposition}
Let $G$ be any weighted directed graph. If a blue vertex $b$ is added to $G$ with in-degree $0$ and out-degree $1$, the incident vertex in $G$ to $b$ is a vertex $v$ with in-weight $d$, and the weight on the new edge is $w$, then
    \[\eptrzf(G+b,b) = \frac{d+w}{w} + \eptrzf(G,v).\]
\end{proposition}

\begin{proof}
As in the unweighted case, if we color $G+b$ starting from $b$, then the first vertex in $G$ that gets colored must be $v$. It takes $\frac{d+w}{w}$ expected rounds until $v$ is colored, and then an additional $\eptrzf(G,v)$ rounds to color the rest of $G$.
\end{proof}

In the next result, we consider the effect of adding a single directed weighted edge between two weighted directed graphs.

\begin{thm}
Let $G, H$ be weighted directed graphs and suppose that $u \in V(G)$ and $v \in V(H)$, where $v$ has in-weight $d$. Let $K$ be the directed graph obtained from $G$ and $H$ by adding a single edge from $u$ to $v$ with weight $w$. Then, for all $x \in V(G)$, we have \[\eptrzf(K,x) \le \frac{d+w}{w}+\eptrzf(G,x)+\eptrzf(H,v).\]
\end{thm}

\begin{proof}
If $x$ is the initial blue vertex, then the expected number of rounds to color all of $G$ is $\eptrzf(G,x)$. Since the new edge has weight $w$ and the remaining in-weight to $v$ is $d$, it takes $\frac{d+w}{w}$ expected rounds for $v$ to be colored after $u$ is colored. After $v$ is colored, the expected number of rounds to color all of $H$ is $\eptrzf(H,v)$. Altogether, it takes at most $\eptrzf(G,x)+\frac{d+w}{w}+\eptrzf(H,v)$ expected rounds to color all of $K$, starting from $x$.
\end{proof}

In the next result, we show that the bound in the last theorem is best possible by exhibiting an infinite family of examples which attain the bound.

\begin{thm}
For all $m, n > 0$, there exist weighted directed graphs $G$ and $H$ and vertices $u, x \in V(G)$ and $v \in V(H)$ such that $\eptrzf(G,x) = m$, $\eptrzf(H,v) = n$, $v$ has in-weight $d$ in $H$, and the weighted directed graph $K$ obtained from $G$ and $H$ by adding an edge from $u$ to $v$ with weight $w$ satisfies \[\eptrzf(K,x) = \frac{d+w}{w}+\eptrzf(G,x)+\eptrzf(H,v).\]
\end{thm}

\begin{proof}
Let $G = \overrightarrow{P_{m+1}}$ have left endpoint $x$ and right endpoint $u$, with all edges of weight $1$. Let $H$ be the weighted directed graph obtained from $\overrightarrow{P_{n+1}}$ with left endpoint $v$ and leftmost edge $(v, v')$ by adding an edge $(v', v)$ of weight $d$ and setting all other edge weights to $1$. Then, $\eptrzf(G,x) = m$, $\eptrzf(H,v) = n$, and $\eptrzf(K) = m+\frac{d+w}{w}+n$.
\end{proof}

\section{RZF on economic data}\label{sec:data}
To illustrate how the methods developed in this paper can be applied beyond stylized graph families, we conclude with a data-driven example using the US Beaureu of Economic Analysis (BEA) input-output accounts \cite{bea_io}. Interpreting each sector as a node and inter-sectoral supply flows as weighted directed edges (including self-loops reflecting internal use), we obtain a real economic network on which RZF can be run without additional modeling assumptions. Computing the EPT from each singleton starting sector produces a propagation-time profile that reflects the inherent structural dependencies of the U.S. production system. Although this example is not intended as an economic analysis, it demonstrates that RZF can be meaningfully extended to empirical networks, capturing asymmetries in flow structure and highlighting sectors whose linkages make them faster or slower origins under the RZF dynamics. This suggests that RZF may be a useful exploratory tool for understanding propagation phenomena in complex real-world systems.

Computation for EPT were conducted using our publicly available implementation \footnote{\url{https://github.com/noah-lichtenberg/rzf_simulation}}, inspired by existing work on using Markov Chains to compute expected propagation time in probabilistic zero forcing \cite{pzf_using_markov}. The implementation works as follows: the BEA input-output data is first converted into a directed graph, with nodes being the 15 sectors and weighted directed edges indicating the monetary value of commodities supplied from one sector to another in 2024, as reported in the BEA input-output tables. We next construct a Markov transition matrix with the transition states being the $2^{15}$ possible colorings of our directed graph, and transition probabilities are derived from the randomized zero forcing rule. Expected propagation times from singleton initial sets are then computed using a dynamic programming approach. 

\begin{figure}[htbp!]
    \centering
    \includegraphics[width=0.95\linewidth]{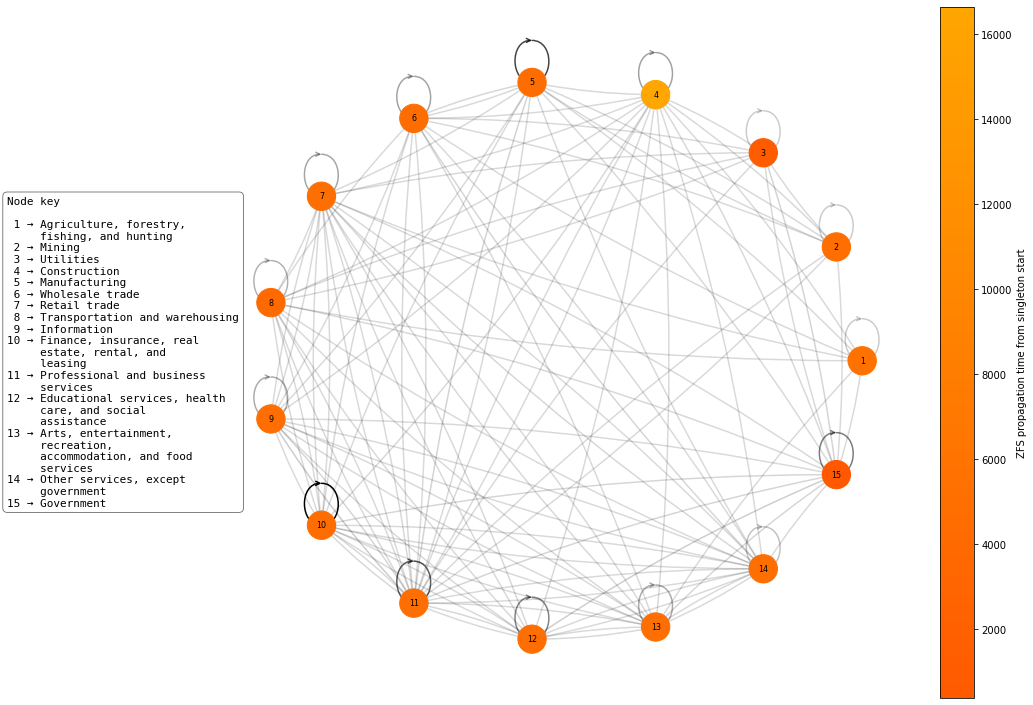}
    \caption{RZF expected propagation times on the BEA input-output sector network. 
    Nodes correspond to economic sectors and are colored by the EPT from a singleton start; 
    edge widths and opacities reflect relative inter-sectoral supply magnitudes.}%
    \label{fig:bea-rzf}
\end{figure}

\begin{table}[ht]
\centering
\caption{Expected propagation time (EPT) from singleton initial sets for BEA sectors}
\label{tab:bea_ept}
\begin{tabular}{clc}
\toprule
Node & Sector & EPT \\
\midrule
1  & Agriculture, forestry, fishing, and hunting & 5679.87 \\
2  & Mining & 4768.77 \\
3  & Utilities & 933.76 \\
4  & Construction & 16652.22 \\
5  & Manufacturing & 4652.18 \\
6  & Wholesale trade & 4684.16 \\
7  & Retail trade & 4777.49 \\
8  & Transportation and warehousing & 4055.41 \\
9  & Information & 4714.20 \\
10 & Finance, insurance, real estate, rental, and leasing & 4667.29 \\
11 & Professional and business services & 4717.40 \\
12 & Educational services, health care, and social assistance & 4755.98 \\
13 & Arts, entertainment, recreation, accommodation, and food services & 4842.59 \\
14 & Other services, except government & 4912.73 \\
15 & Government & 385.06 \\
\bottomrule
\end{tabular}
\end{table}

For the BEA sector network, the computed $\eptrzf$ values vary widely across sectors (see Figure~\ref{fig:bea-rzf}). We first note that the EPT values are extremely high across all sectors, which can be attributed to the large edge weights on the self-loops, indicating that many sectors rely mostly on themselves rather than other sectors. Notably, the government sector stands out with the smallest expected propagation time (approximately $385$), indicating that a shock originating in government would spread through the economy more quickly than one from any other sector. The utilities sector is also highly central with a very low $\eptrzf$ (under $1000$). In contrast, the construction sector yields by far the largest propagation time (on the order of $1.7\times 10^4$), suggesting that a disturbance starting in construction remains comparatively localized for a long time. Most other industries fall in an intermediate range of $\eptrzf$ (roughly $4\times 10^3$-$5\times 10^3$); for example, manufacturing, wholesale trade, retail, information, finance, and professional services all have EPT values on the order of $4.6\times 10^3$-$4.8\times 10^3$, while agriculture is modestly higher (about $5.7\times 10^3$). These magnitudes reflect the varying degrees of connectivity and influence each sector has in the input-output network.

Importantly, $\eptrzf$ in this context provides a quantitative measure of peripherality for each sector in the production network. A larger expected propagation time means the sector is more peripheral - a shock originating there diffuses slowly to the rest of the economy - whereas a smaller $\eptrzf$ implies the sector is more structurally central, with shocks from that sector reaching other sectors more rapidly. In other words, the inverse of $\eptrzf$ can be interpreted as a centrality index under the RZF dynamic: sectors like government and utilities, which have high $1/\eptrzf$ (due to low propagation times), occupy central positions in the network of inter-industry linkages. Conversely, sectors such as construction (with a very low $1/\eptrzf$ due to its huge propagation time) are situated on the network’s periphery. This perspective complements traditional static centrality measures by focusing on dynamic risk propagation behavior. It underscores, for instance, that government’s extensive linkages (e.g. procurement from many industries) and utilities’ fundamental input role make them pivotal for rapid contagion, while construction’s more self-contained supply chain position renders it a slower conduit for propagating shocks. Overall, examining $\eptrzf$ across all nodes highlights a pronounced asymmetry in the economy’s connectivity: some sectors serve as fast “hubs” for spreading disturbances, whereas others act as bottlenecks or dead-ends, suggesting that RZF-based analysis can be used for identifying structurally critical sectors.

\section{Conclusion and future directions}\label{sec:conclusion}

In this paper we introduced \emph{randomized zero forcing} (RZF), a stochastic
color-change process on directed graphs in which the probability that a white
vertex turns blue depends on the fraction of its in-neighbors that are blue,
and we extended the model to weighted directed graphs by replacing fractions of
neighbors with fractions of incoming weight.  We established basic monotonicity
properties (with respect to enlarging the initial blue set, and with respect to
increasing weights on edges out of initially blue vertices), and we identified
weight normalizations under which the weighted process reduces to the unweighted
process.  We then computed exact propagation times or sharp asymptotics for a
variety of graph families (including arborescences, stars, paths, cycles, and
balanced spiders), and we proved extremal bounds for unweighted directed graphs
in terms of parameters such as the number of edges, order, maximum degree, and
radius, with constructions showing sharpness up to constant factors.  Finally,
we discussed how $\eptrzf$ behaves under several graph operations and illustrated
how the weighted model can be run on an empirical input-output network.

Our results suggest several natural directions for further study. In deterministic zero forcing, throttling problems quantify the tradeoff between the size of the initial blue set and the time needed to force the whole graph.  Because $\eptrzf(G,S)$ is monotone nonincreasing in $S$, an analogous optimization problem is well-posed for RZF.  For example, one can define an \emph{RZF throttling number} by \[
        \thrzf(G)\;:=\;\min_{S\subseteq V(G)}\Bigl(|S|+\eptrzf(G,S)\Bigr),
    \]
or consider constrained variants such as minimizing $\eptrzf(G,S)$ subject
to $|S|\le k$, or minimizing $|S|$ subject to $\eptrzf(G,S)\le T$. Throttling has previously been investigated for standard zero forcing \cite{butler}, skew zero forcing \cite{skewth}, positive semidefinite zero forcing \cite{psdth}, power domination \cite{pdthrot}, the cop-versus-robber game \cite{cr2,cr1,damage}, the cop-versus-gambler game \cite{gth2,gth1}, and metric dimension \cite{thdim}. It would
be especially interesting to determine $\thrzf(G)$ (or sharp bounds for it) on the same families studied here (paths, cycles, trees, complete graphs), and to compare how the optimal tradeoff depends on directionality and (in the weighted setting) on the distribution of incoming weights.

A compelling extension for applications is to allow recovery, i.e., blue vertices can revert to white.  One simple variant is: after each round of RZF updates, each blue vertex independently reverts to white with probability $\rho\in(0,1)$ (or with a vertex-dependent rate $\rho_v$).  This produces a finite-state Markov chain with potentially rich behavior: depending on parameters, the all-blue state may no longer be absorbing, and one may instead study (i) the probability of ever reaching all-blue, (ii) expected hitting times conditional on reaching all-blue, and/or (iii) the stationary distribution and mixing time.  Developing techniques that replace monotone couplings, which are central in the present paper, would be a key step in understanding this regime. Vertex reversion was investigated for probabilistic zero forcing in \cite{brennan}.

Throughout the paper we focus on expectations, but several arguments (for example, the spider analysis and joined-graph bounds) naturally raise questions about variance and higher moments.  For a given family, one can ask for asymptotics of $\Var(\tau)$ where $\tau$ is the propagation time, and for tail bounds of the form $\Prob(\tau \ge t)$ or $\Prob(|\tau-\E\tau|\ge t)$.  Such results would clarify when the expectation is representative of typical behavior, and could enable sharper bounds for graph operations that involve maxima over independent or weakly dependent sub-processes.

The weighted model is motivated by input-output and other flow networks.
This suggests optimization problems in which one modifies weights or adds edges to \emph{slow} or \emph{accelerate} propagation under constraints.
For example, given a budget to increase or decrease a small set of incoming weights, which modifications maximize $\eptrzf(G,S)$ (to increase resilience) or minimize it (to model fast diffusion)?  The monotonicity results in the basic section provide starting tools for such comparative statics, but a broader theory for general weight perturbations (not restricted to edges out of $S$) would be especially useful.

Many real-world networks evolve over time, with edges or weights that change in response to the state of the system. Extending RZF to temporal graphs or
adaptive-weight models, where weights respond dynamically to the coloring
process itself, could significantly broaden its applicability. In such
settings, expected propagation time may interact nontrivially with feedback
mechanisms, creating new phenomena absent in static graphs.

The BEA example illustrates that $\eptrzf$ can be computed on real weighted directed networks and used as a dynamic, process-based notion of centrality. Important next steps include understanding sensitivity to measurement error and missing edges as in \cite{brz23}, developing principled coarse-graining (aggregation of vertices) that preserves propagation behavior, and comparing $\eptrzf$-based rankings to classical centrality measures.  These questions are particularly relevant when the network is only partially observed or evolves over time.

\section*{Acknowledgments}
This work was initiated at an AIM SQuaRE workshop on ''Systemic Importance of Entities and Risk Propagation in Input-Output Graphs.'' We thank AIM for providing a supportive research environment. GPT-5 was used for proof development, exposition, and revision.

\end{document}